\documentclass[12pt, a4paper]{article}

\usepackage[utf8]{inputenc}
\usepackage[T1]{fontenc}
\usepackage{amsthm}
\usepackage{amsmath}
\usepackage{amsfonts}
\usepackage{amssymb}
\usepackage{caption}
\usepackage{float}
\usepackage{graphicx}
\usepackage{ifthen}
\usepackage{tikz}
\usepackage{tikz-cd}
\usepackage{subcaption}
\usepackage{longtable}
\usepackage{fullpage}
\usepackage{textcomp}
\usepackage{mathtools}
\usepackage[all]{xy}
\usepackage{makecell}
\usepackage[hidelinks]{hyperref}
\usepackage{setspace}
\usepackage[a4paper]{geometry}
\usepackage{color}
\usepackage{amsmath}
\usepackage{algorithm}
\usepackage{afterpage}
\usepackage[noend]{algpseudocode}
\usepackage{chngcntr}

\usetikzlibrary{patterns}
\usetikzlibrary{intersections}

\newtheorem{theorem}{Theorem}
\newtheorem{construction}[theorem]{Construction}
\newtheorem{proposition}[theorem]{Proposition}
\newtheorem{lemma}[theorem]{Lemma}
\newtheorem{corollary}[theorem]{Corollary}

\newtheorem{conjecture}[theorem]{Conjecture}
\newtheorem{definition}[theorem]{Definition}
\newtheorem{example}[theorem]{Example}

\counterwithin{theorem}{section}
\counterwithin{equation}{section}
\counterwithin{figure}{section}
\counterwithin{table}{section}

\newcommand \crosssize{0.02}
\newcommand \NodalCurve[4] {\draw (#1, #2+2*#4) .. controls (#1+#3,#2-#4*4) and (#1+#3, #2+4*#4) .. (#1,#2-#4*2);}
\newcommand \Lawrence [1] {{#1}}
\newcommand \CharSheaf [1] {\overline{\mathcal{M}}_{#1}}

\algnewcommand{\LeftComment}[1]{\Statex \(\triangleright\) #1}

\def \Spec {\operatorname{Spec} \; }

\def \LogSpec {\operatorname{Spec}_{log} \; }
\def \PP {\mathbb{P}}
\def \NN {\mathbb{N}}

\def \CC {\mathbb{C}}
\def \AA {\mathbb{A}}
\def \RR {\mathbb{R}}
\def \ZZ {\mathbb{Z}}

\def \Cone {\mathcal{C}one}
\def \Bad {Bad}

\def \MacBeth {\mathcal{D}}
\def \Coker {\mathcal{C}oker}
\def \Ker {\mathcal{K}er}
\def \Hom {\mathcal{H}om}
\def \Mor {\operatorname{Hom} \; }
\def \Geom {\operatorname{Geom} \; }
\def \StableModuli {\overline{\mathcal{M}}}
\def \LogModuli {\overline{\mathcal{M}}^\dagger}

\def \fib {\textrm{fib}}
\def \tr {\textrm{tr}}
\def \contr {\textrm{contr}}
\def \SheafExt {\mathcal{E}xt}
\def \Hilb {\mathcal{H}ilb^2 (\mathbb{P}^1)}
\def \MTwoOne {\StableModuli_{0,0} (\PP^1, 2)}

\begin{document}

\title{Convergence of the mirror to a rational elliptic surface}
\author{Lawrence Jack Barrott}
\date{}

\maketitle

\begin{abstract}

The construction introduced by Gross, Hacking and Keel in~\cite{MirrorSymmetryforLogCY1} allows one to construct a formal mirror family to a pair $(S,D)$ where $S$ is a smooth rational projective surface and $D$ a certain type of Weil divisor supporting an ample or anti-ample class. In that paper they proved two convergence results. Firstly that if the intersection matrix of $D$ is not negative semi-definite then the family they construct lifts to an algebraic family. Secondly they prove that if the intersection matrix is negative definite then their construction lifts along certain analytic strata on the base, and then over a formal neighbourhood of this. In the original version of that paper they claimed that if the intersection matrix were negative semi-definite then family in fact extends over an analytic neighbourhood of the origin but gave an incorrect proof.

In this paper we correct this error. We explain how the general Gross-Siebert program can be used to reduce construction of the mirror to such a surface to calculating certain relative Gromov-Witten invariants. We then relate these invariants to the invariants of a new space where we can find explicit formulae for the invariants. From this we deduce analytic convergence of the mirror family, at least when the original surface has an $I_4$ fibre.

\end{abstract}

\section*{Introduction}

The geometry of surfaces often is a rich playground in which to test new conjectures and constructions. In~\cite{MirrorSymmetryforLogCY1} the authors introduced a construction of a mirror family to certain pairs $(S,D)$, called Looijenga pairs. This built upon work of~\cite{AurouxKatzarkovOrlov} who had made predictions of the mirror families to various del Pezzo surfaces. These del Pezzo surfaces are the smooth Fano surfaces, they have ample anti-canonical class. However both constructions make sense when $-K_S$ is not ample, but merely nef. This corresponds to the case where $S$ is a rational elliptic surface and $D=F$ a rational fibre.

There is another direction from which to approach this question. Gross, Hacking and Keel were trying to understand mirror symmetry over the moduli space of $K3$ surfaces. Recall that there is a classification of the degenerations of a $K3$ surface.

\begin{theorem}[Structure of $K3$ degenerations]

Let $\pi:X \rightarrow \AA^1$ be a semi-stable family such that the fibre $X_t$ for $t \neq 0$ is a smooth $K3$ surface and the canonical bundle of the family is trivial. Then the central fibre $X_0$ is isomorphic to one of the following types

  \begin{itemize}
  \item[I] A smooth $K3$ surface.
  \item[II] A chain $S_1, \ldots S_n$ of surfaces such that $S_1$ and $S_n$ are rational, the other components are ruled surfaces over elliptic
curves and $S_i \cap S_{i+1}$ are smooth elliptic curves.
    \item[III] A complex of rational surfaces $S_1 \ldots S_n$ such that $V_i \cap V_j$ are smooth rational curves. The dual complex gives a triangulation of $S^2$.
  \end{itemize}

In the second case one can perform a birational modification of the family to produce a type II degeneration where the central fibre is the union of only two components. In this case the class $S_1 \cap S_2$ is an anti-canonical curve on each component.

\begin{proof}

  See~\cite{DegenerationsOfK3SurfacesAndEnriquesSurfaces}.
  
\end{proof}  

\end{theorem}

So one could hope to apply their construction to a point on this boundary, a type II degeneration, and then extend out the family constructed over the entire moduli space. This paper therefore answers first half to constructing the mirror to a type II degeneration. There is much more work to be done before one can extend this example to the type II fibre.

This paper begins by recalling the general philosophy of the Gross-Siebert program, explaining the appearance of scattering diagrams and broken lines. This should be seen as a continuation of my work in~\cite{ExplicitEquationsForMirrorFamiliesToLogCalabiYauSurfaces} which features a slightly introduction to this philosophy. We explore how the combinatorial nature of these objects limits the data we must calculate. In our case this reduces the problem to calculating the number of rational curves meeting the fixed fibre $F$ with tangency orders one and two.

Having done that we discuss work of Bryan and Leung who proved that the number of curves in a class $E + nF$ for $E$ a section appear as coefficients of a known modular form. These include all the curves tangent to $F$ to order one. We extend their results to our case, where we are interested in relative invariants rather than absolute invariants.

With this warm up done we move to the main challenge of this paper, counting the curves tangent to the boundary at a single point with tangency order two. To do this we construct a threefold whose geometry allows us to lift curves from $S$. This threefold is a nef complete intersection in a toric variety, hence we can construct the Gromov-Witten theory via Givental's $I$ and $J$ functions. Proving relations between the invariants on $S$ and on the threefold requires that we delve into the obstruction theories on both, applying techniques from~\cite{VirtualPullbacks} and~\cite{VirtualPushforward}. 

Having done this we then return to the original construction and deduce that it converges in a neighbourhood of the origin, exactly as desired.

Throughout this $S$ will denote a rational elliptic surface, which by Lemma IV.1.2 of~\cite{GeometryofEllipticSurfaces} is isomorphic to the blowup in the nine points lying on the intersection of two cubics in $\PP^2$. The Chow group $A_1 (S)$ is spanned by $H$, the pullback of a hyperplane, and $E_1, \ldots, E_9$, the exceptional curves of the blowup. The Chow groups $A_0(S)$ and $A_2(S)$ are both one dimensional, spanned by a point and a fundamental class respectively. $F$ will denote a general fibre, whilst $F_0$ will denote an $I_4$ fibre, a cycle of four $-2$-curves. We will write $S^\dagger$ for $S$ together with the divisorial log structure coming from $F_0$. By an inclusion of components $i:X \rightarrow Y$ we mean an isomorphism from $X$ to a union of connected components of $Y$.

This work was the bulk of my thesis project, and I must thank Mark Gross for suggesting this problem to me. Paul Hacking and Sean Keel also contributed both in writing their paper with Mark Gross and through discussions and motivation for the various different parts of the construction. Tom Coates gave a wonderfully clean description of the work of Givental which greatly improved my understanding of how those techniques work. My thesis was funded by an Internal Graduate Studentship of Trinity College Cambridge, an extension grant from the Department of Pure Maths and Mathematical Statistics in Cambridge, and a research studentship from the Cambridge Philosophical Society. 

\section{The Gross-Siebert program}

The Gross-Siebert program began as an attempt to make rigorous geometric constructions from the SYZ conjecture. 

\begin{conjecture}[The SYZ conjecture]

Let $X, \check{X}$ be a mirror pair of \emph{Calabi-Yau varieties}. Then there is an affine manifold with singularities $B$ and maps $\phi: X \rightarrow B$, $\check{\phi}: \check{X} \rightarrow B$ which are dual special Lagrangian fibrations. This is shown schematically on the next page.

\begin{figure}[h]
\centering
\begin{tikzpicture}
\draw (-6,-2) -- (-2,-2) -- (-2,2) -- (-6,2) -- (-6,-2);
\draw (2,-2) -- (6,-2) -- (6,2) -- (2,2) -- (2,-2);
\draw (-2.2, -1.8) node {$X$};
\draw (5.8, -1.8) node {$\check{X}$};
\draw (-4,0) ellipse (0.2 and 2);
\draw (-4,1) .. controls (-3.93,0) .. (-4,-1);
\draw (4,0) ellipse (0.2 and 2);
\draw (4,1) .. controls (4.07,0) .. (4,-1);
\draw[->] (-4,-2.2) -- (-2,-4);
\draw[->] (4,-2.2) -- (2,-4);
\draw (-4.4, 0) node {$T$};
\draw (3.5, 0) node {$T^\vee$};
\draw (-4, -3) node {$\phi$};
\draw (4, -3) node {$\phi^\vee$};
\draw[thick] (-2,-5) -- (-1, -4.5) -- (1, -5) -- (2, -6) -- (1.5, -6.5) -- (0, -7) -- (-2,-5);
\draw (-2,-5) -- (-1, -5.2) -- (-1, -4.5);
\draw (-1, -5.2) -- (0.6, -5.8) -- (1,-5);
\draw (0.6, -5.8) -- (2,-6);
\draw (0.6, -5.8) -- (0.6,-6) -- (1.5,-6.5);
\draw (0.6, -6) -- (0,-7);
\end{tikzpicture}
\end{figure}

\end{conjecture}

In the case of toric varieties this conjecture is tractible, the base is a polytope for the variety and the fibration the momentum map. This map degenerates over the boundary in a very controlled manner, showing that this fibration is certainly not trivial. Gross and Siebert suggested that even without knowing the exact form of the fibration one could still recover enough information to reconstruct the mirror family. To do this one must reincorporate information about the singularities via a scattering diagram.

The construction of~\cite{MirrorSymmetryforLogCY1} studies pairs $(S,D)$ where $S$ is a smooth rational surface and $D$ a cycle of rational curves (otherwise known as a \emph{Looijenga pair}). It states that the base for the fibration of should be the \emph{dual intersection complex of $D$} which we construct below. This is inspired by the toric case where this complex is the base of the momentum map. Of course this base need not embed into $\RR^2$, rather it has a singularity at the origin. This is precisely as expected from the SYZ picture and is where interesting geometry can enter the picture.

\begin{construction}

Let $D = D_1 \cup D_2 \cup \ldots \cup D_n$ be a cycle of rational curves on a smooth surface $S$ such that $D_i \cap D_j$ is a single point just when $i$ and $j$ differ by 1 mod $n$ and otherwise is empty (in the case $n=2$ we relax this to saying that there are two points in the unique intersection $D_1 \cap D_2$). Then $\Delta_{(S,D)}$ contains precisely one zero-dimensional cell 
$\{0\}$ corresponding to the interior $S \setminus D$. For each component $D_i$, $\Delta_{(S,D)}$
contains a one-dimensional cone with $v_i$ its primitive generator.
Attach the zero-dimensional cone as 0 inside each of these rays. Now introduce a two-dimensional cone in $\Delta_{(S,D)}$ for each intersection point of $D_i \cap D_j$, spanned by $v_i$ and $v_j$. This produces the \emph{dual intersection complex} $\Delta_{(S,D)}$ as a cone complex but it carries more structure. 
We write $B$ to be the underlying topological space of the cone complex
$\Delta_{(S,D)}$. It is homeomorphic to $\RR^2$. We give $B\setminus\{0\}$
an affine manifold structure by defining an affine coordinate chart by embedding
the union of the cones $\RR^{\geq 0} v_{i+1} \oplus \RR^{\geq 0} v_i $ and $\RR^{\geq 0} v_{i} \oplus \RR^{\geq 0} v_{i-1}$ into $\RR^2$ via the relations $v_{i-1} \mapsto (1,0)$, $v_i \mapsto (0,1)$ and $v_{i+1} \mapsto (-1,-D_i ^2)$. This expresses $\Delta_{(S,D)}$ not just as a complex of sets but an affine manifold with singularities (indeed a single singularity at the origin).

\end{construction}

Let us draw out in Figure 1 this complex, or rather its universal cover, in the case where $S$ is a rational elliptic surface and $D = D_1 \cup \ldots \cup D_4$ is a cycle of $-2$-curves, i.e. an $I_4$ fibre.

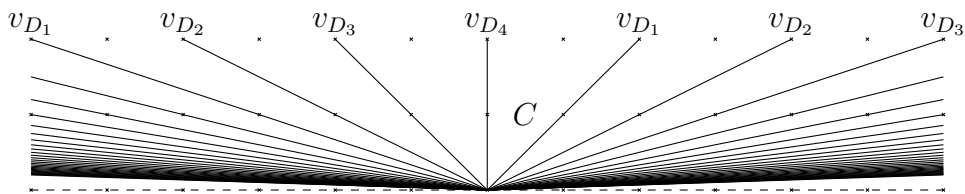
\begin{figure}[h]
\centering
\begin{tikzpicture}
\foreach \x in {-6,-5,-4,-3,-2,-1,0,1,2,3,4,5,6}
\foreach \y in {0,1,2}
{
\draw (\x+\crosssize,\y+\crosssize) -- (\x-\crosssize, \y-\crosssize);
\draw (\x+\crosssize,\y-\crosssize) -- (\x-\crosssize, \y+\crosssize);
}
\foreach \n in {-6,-4,-2,0,2,4,6}
{
\draw (0,0) -- (\n, 2);
}
\foreach \n in {8, 10, 12, 14, 16, 18, 20, 22, 24, 26,28,30,32,34,36,38,40,42,44,46,48,50,52,54,56,58,60}
{
\draw (0,0) -- (6, 12 / \n);
\draw (0,0) -- (-6, 12 / \n);
}
\node at (0.5,1) {$C$};
\node at (0,2.2) {$v_{D_4}$};
\node at (2,2.2) {$v_{D_1}$};
\node at (4,2.2) {$v_{D_2}$};
\node at (6,2.2) {$v_{D_3}$};
\draw [dashed] (0,0) -- (6,0);
\node at (-2,2.2) {$v_{D_3}$};
\node at (-4,2.2) {$v_{D_2}$};
\node at (-6,2.2) {$v_{D_1}$};
\draw [dashed] (0,0) -- (-6,0);
\end{tikzpicture}
\caption{The universal cover of the dual intersection complex for $S^\dagger$.}
\label{fig:DualIntersectionRES}
\end{figure}

In the toric case one can construct the mirror family via the Mumford degeneration. We give the general construction first, then apply it to our setting and then explain how to apply it to the toric case.

\begin{example}

  Let $B \subset N_\RR$ be a lattice polyhedron, $\mathcal{P}$ a lattice polyhedral decomposition of $B$ and $\phi: B \rightarrow \RR$ a strictly convex piecewise linear integral function, i.e. it is strictly convex between different cells of $\mathcal{P}$. One takes the graph over $\phi$ to produce a new polyhedron

\[\Gamma (B, \phi) := \{ (n, r) \in N_\RR \oplus \RR \mid n \in B, r \geq \phi (n)\}\]

This produces a lattice polyhedron unbounded in the positive direction on $\RR$. To construct a family from this we perform a cone construction, let $C (\Gamma)$ be the closure of the cone over $\Gamma (B, \phi)$:
\[
C(\gamma) = \overline{\{(n,r_1,r_2)\in N_{\RR}\oplus\RR\oplus\RR\,|\,
(n,r_1)\in r_2 \Gamma(B,\phi)\}}.
\]
 This carries an action of $\RR^+$ given by translating the second component of $\Gamma(B, \phi)$. The integral points of $C (\Gamma)$ form a graded monoid and we can take $Proj$ of this to produce a variety projective over $\AA^1$. We write this $\PP_{\Gamma (B,\phi)}$. The general fibre is a toric variety, 
whilst the fibre over the origin is a union of toric varieties.

\end{example}

Gross, Hacking and Keel modify this construction by incorporating singularities into the base. Outside the toric case we lack global coordinates for the Mumford construction. Instead they introduce a bundle with local coordinates.

\begin{definition}
\label{def:sheaf-of-monoids}
  Let $(S,D)$ be a Looijenga pair with associated dual intersection complex $\Delta_{(S,D)}$ and let $\eta: NE (S) \rightarrow M$ be a homomorphism of monoids. We want to construct a multi-valued piecewise linear function on $\Delta_{(S,D)}$ which will bend only along the one-cells. This will be a collection of
piecewise linear functions defined on open subsets of $B$ which differ by linear
functions on overlaps. Such functions are determined by their bends at
one-cells, which are encoded as follows. For a
one-cell $\tau=\RR^{\geq 0} v_i$, choose an orientation $\sigma_+$ and 
$\sigma_-$ of the two two-cells separated by $\tau$ and let $n_\tau$ be the unique primitive linear function positive on $\sigma_+$ and annihilating $\tau$. We want
to construct a representative $\phi_i$ for $\phi$ on $\sigma_+\cup \sigma_-$.  Writing $\phi_+$ and $\phi_-$ for the linear function defined by $\phi_i$ on $\sigma_+$ and $\sigma_-$, the function $\phi_i$ is then defined up to a linear
function by the requirement that
\[ \phi_+ - \phi _- = n_\tau \otimes \eta ([D_i])\]
Such a function is convex in the sense of~\cite{MirrorSymmetryforLogCY1} Definition 1.11, and one says that it is \emph{strictly convex} if $\eta([D_i])$ is not invertible for any $i$. 

This function determines a $M^{gp}\otimes\RR$-torsor $\PP$ as defined 
in~\cite{GrossSiebert}, Construction 1.14, on $B\setminus \{0\}$
which is trivial on each $(\sigma_+\cup\sigma_-)\setminus\{0\}$, i.e.,
is given by $((\sigma_+\cup\sigma_-)\setminus \{0\})\times (M^{gp}\otimes\RR)$.
These trivial torsors are glued on the overlap of two adjacent such sets,
namely on $\RR^{\ge 0} v_{i} \oplus\RR^{\ge 0}v_{i+1}$, via the map
\[ (x,p) \mapsto (x, p + \phi_{i+1} (x) - \phi_i (x))\]
which induce isomorphisms on the monoid of points lying above $\phi$.
This construction is designed to allow us to run the Mumford construction locally, even if we cannot run it globally. 
Write
$\pi: \PP \rightarrow B\setminus\{0\}$, and we write $\mathcal{M}$ for the 
sheaf $\pi_* (\Lambda _{\PP})$, bearing in mind that $\PP$ also has the
structure of an affine manifold. Then there is a canonical exact sequence
  \[ 0 \rightarrow \underline{M^{gp}} \rightarrow \mathcal{M} \xrightarrow{r} \Lambda_{B} \rightarrow 0 \]
  We write $r$ for the second map in this sequence. This will not be mentioned again until we define the canonical scattering diagram so keep this in mind until then.
Furthermore, \cite{GrossHackingSiebert}, Definition 1.16 gives a subsheaf
of monoids $\mathcal{M}^+\subset \mathcal{M}$.
 
\end{definition}

If one performs this with a toric variety relative to its boundary and pairs the function $\phi$ with an ample class then one obtains a height function for the Mumford degeneration. In particular the general fibre of this family will be isomorphic to the original toric variety. At the moment this lacks any data about the singularities of the SYZ fibration. Without this data we cannot hope to correctly recreate the mirror. Let us recall how the singularities appear.

The authors in~\cite{MirrorSymmetryforLogCY1} claim that the mirror should be affine, the spectrum of the symplectic cohomology of $(S,D)$. This is not explicated there, and indeed is not true in general. Daniel Pomerleano has written down three dimensional examples whose ``mirror family'' is generically not smooth, one must pass to a resolution of this family to obtain the actual mirror family. In any case we accept this and ask how to describe the symplectic cohomology.

By definition this ring counts Maslov index two disks with boundaries lying on Lagrangians. Furthermore given a Maslov index zero disk we may glue it to a standard index two disk. Now the image of one of these index two disks under the fibration is a topological space retracting onto a piecewise linear skeleton. This skeleton is a balanced piecewise linear graph on $B$ with three legs, one ending on the central singularity and two more passing off to infinity. The linear structure breaks down near points where a Maslov index zero disk attaches. This process goes by the name \emph{tropicalisation}.

Gross, Hacking and Keel take this definition and study all possible trivalent tropical curves, piecewise linear balanced graphs, on $B$. They attempt to reformulate the Maslov index zero disks algebraically as a scattering diagram on $B$.

\begin{definition}

  Let $B$ be an affine manifold with a single singularity, with 
$B$ homeomorphic to $\RR^2$ with singularity at the origin, so that $B^*:=
B\setminus \{0\}$ is an affine manifold.
Let $\mathcal{M}$ be a locally constant sheaf of abelian groups
on $B^*$ with a subsheaf of monoids $\mathcal{M}^+\subset\mathcal{M}$
and equipped with a map $r:\mathcal{M} \rightarrow \Lambda_B$.
Let $\mathcal{J}$ be a sheaf of ideals in $\mathcal{M}^+$ with stalk $\mathcal{J}_x$ maximal in $\mathcal{M}^+_x$ for all $x \in B^*$.
Let
$\mathcal{R}$ denote the sheaf of rings locally given by the completion of 
$k[\mathcal{M}^+]$ at $\mathcal{J}$. 
 A \emph{scattering diagram with values in the pair $(\mathcal{M}, \mathcal{J})$ on $B$} is a function $f$ which assigns to each rational ray from the origin an section of the restriction of $\mathcal{R}$ to the ray.
We require the following properties of this function:

  \begin{itemize}
    \item For each $\mathfrak{d}$ one has $f(\mathfrak{d}) = 1$ mod 
$\mathcal{J}|_{\mathfrak{d}}$.
    \item For each $n$ there are only finitely many $\mathfrak{d}$ for which $f(\mathfrak{d})$ is not congruent to $1$ mod $\mathcal{J}|_{\mathfrak{d}}^n$. These $\mathfrak{d}$ are called \emph{walls}.
    \item For each ray $\mathfrak{d}$ and for each monomial $z^{p}$ appearing in $f(\mathfrak{d})$ one has $r(p)$ tangent to $\mathfrak{d}$. A line for which $r(p)$ is a positive generator of $\mathfrak{d}$ for all $p$ with $c_p\not=0$
 is called an \emph{incoming ray}. If instead
$r(p)$ is a negative generator of $\mathfrak{d}$ for all $p$ with $c_p\not=0$,
it is called an \emph{outgoing ray}.
  \end{itemize}
  
  We denote such an object by the tuple $(B, f, \mathcal{M}, \mathcal{J})$. We say that $(B,f,\mathcal{M},\mathcal{J})$ is obtained from $(B,f',\mathcal{M},\mathcal{J})$ by \emph{adding outgoing rays} if for each ray $\mathfrak{d}$ one can write $f(\mathfrak{d}) = f' (\mathfrak{d}) (1 + \sum (c_p z^p))$ where for each monoid element $p$ with $c_p\not=0$ the vector $-r(p)$ is a generator of $\mathfrak{d}$.

\end{definition}

In our case the sheaf of monoids $\mathcal{M}^+$ will be as given in
Definition \ref{def:sheaf-of-monoids}, with the monoid $M$ being
 a finitely generated sharp submonoid of $H_2(S,\ZZ)$
containing $NE(S)$, the monoid generated by effective curves on $S$. Being sharp means the only invertible element of $M$ is the identity element, and so the maximal ideal is just the complement of the identity. We introduce the choice of scattering diagram, the \Lawrence{\emph{canonical scattering diagram},} and then motivate and define each term appearing. A ray $\mathfrak{d} = (a v_i + b v_{i+1}) \RR^{\geq 0}$ specifies a blow up of $S$ given by refining the fan until $\mathfrak{d}$ is a one-cell and all the two-cells are integrally isomorphic to $(\RR^{\geq 0})^2$. The ray $\mathfrak{d}$ then corresponds to a component $C$ in the inverse image of $D$. To this ray then we assign the power series
\[f(\mathfrak{d}) := \exp \left[ \sum_\beta k_{\beta} N_\beta z^{\eta (\pi_* (\beta)) - k_\beta m_{\mathfrak{d}}'}\right],\]
where $m_{\mathfrak{d}}'$ is the unique lift of a primitive outward
pointing tangent vector $m_{\mathfrak{d}}$ to $\mathfrak{d}$ not contained
in the relative interior of $\mathcal{M}|_{\mathfrak{d}}$.
The number $N_{\beta}$ counts the number of relative curves mapping to $(S,D)$ tangent to $C$ to maximal order $k_\beta$ at a single point as outlined in~\cite{MirrorSymmetryforLogCY1} at the beginning of section 3.
 By the Riemann-Roch formula this is of dimension
\[ dim \: S - 3 - K_S. \beta - k_\beta + 1 = 0.\]
and applying~\cite{GromovWittenInvariantsInAlgebraicGeometry} this produces a Gromov-Witten invariant. This construction may also be run using logarithmic Gromov-Witten invariants rather than these blow ups, see~\cite{LogarithmicGromovWittenInvariants} for the definition, and we take this approach later, defining the relevant theory.

The tropicalisation of a Maslov index two disk is trivalent and therefore we could study each leg separately. A broken line formalises one leg of this picture, the full picture arises once we glue two of these together

\begin{definition}

  A \emph{broken line from $v\in B(\ZZ)$ to $P \in B$} on a scattering diagram $(B, f, R, J)$ is a choice of piecewise linear function $l: \RR^{\leq 0} \rightarrow B$ and a map $m: \RR^{\leq 0} \rightarrow \prod_{t\in\RR^{\leq 0}}
\mathcal{R}_{l(t)}$ such that the following hold:

  \begin{itemize}
  \item $l$ has only finitely many points where it is non-linear and these only occur where $l$ maps into a rational ray of $B$.
  \item $l (0) = P$.
  \item $l(t)$ lies in the same cone as $v$ and is parallel to $v$ for all $t$ sufficiently negative.
  \item $l$ does not map an interval to a ray through the origin.
    \item $m(t)\in \mathcal{R}_{l(t)}$ for all $t$ and is a monomial in
this ring, written as $c_t z^{m_t}$. Further, on each domain of linearity
of $l$, $m(t)$ is given by a section of $\mathcal{R}$ pulled back to 
this domain of linearity.
  \item For $t$ very negative, $m_t\in \mathcal{M}^+_{l(t)}$ is the
unique element not lying in the interior of this monoid satisfying
$r(m_t) = v$.
  \item $r(m(t_0)) = -{ \partial l \over \partial t}\big|_{t=t_0}$ wherever $l$ is linear.
  \item $m$ only changes at those points where $l$ is non-linear.
  \item Let $t \in \RR_{\leq 0}$ be a point where $l$ is non-linear. We write $\partial(l_+)$, $\partial(l_-)$, $m_+$ and $m_-$ for the values of $\partial l \Lawrence{/} \partial t$ and $m$ on either side of $t$. Suppose that $l(t)$ lies on a ray $\mathfrak{d}$ with primitive normal vector $n_{\mathfrak{d}}$, negative on $\partial (l_-)$. Then $m_+$ is a monomial term of $m_- f_{\mathfrak{d}}^{\langle n_{\mathfrak{d}}, \partial (l_-)\rangle}$.
\end{itemize}

\end{definition}

The key data we will need is the count of pairs of pants, which are expressed in terms of broken lines as pairs of broken lines $(l_P, m_P)$  and $(l_Q, m_Q)$ from $P$ and $Q$ respectively to an irrational point near to $R$ such that 
$(\partial l_P/\partial t)|_{t=0} + (\partial l_Q/\partial t)_{t=0} + R=0$. Let $T_{P,Q \rightarrow R}$ denote the set of such pairs.

Now to construct an algebraic version of the symplectic cohomology we follow~\cite{CanonicalBasesForClusterAlgebras} and for each integral point $P$ in $B$ introduce a symbol $\vartheta_P$. As a $k$-vector space the ring $QH (\check{W})$ is freely generated by the $\vartheta_P$. We take the content of~\cite{MirrorSymmetryforLogCY1} Theorem 2.34 as the definition of the product of $\vartheta_P$ and $\vartheta_Q$, so this product is equal to

\[ \sum_{R} \sum_{((l_P, m_P), (l_Q, m_Q)) \in T_{P,Q \rightarrow R}} m_P (0) m_Q (0) \vartheta_R\]

This is analogous to the approach of ~\cite{GrossHackingSiebert}. A key result of ~\cite{MirrorSymmetryforLogCY1} is that for fixed order $J^n$ and for a consistent scattering diagram in the sense of
~\cite{MirrorSymmetryforLogCY1}, Definition 2.26, this does not depend on the choice of irrational point near $R$. This consistency property will hold
in particular for the canonical scattering diagram.
This sum need not terminate and indeed will in general only produce a power series. However, in the case that $D$ supports an ample divisor, this power series will in fact be a polynomial.

Now let us forget everything about the scattering diagram and suppose that it is as bad as possible, so it has rays in every direction. We can use the integral structure to limit the possible terms appearing in the product. Let us explain how to apply this idea in the case of a rational elliptic surface.

\begin{example}

  We apply the balancing condition to limit the bends which can occur. Consider the $y$-coordinate of a point of $B$, $y(P)$. This extends to the tangent bundle of $B$ and we write $y(v)$ for this extension. Importantly for us $y(-)$ is positive on points of $B$, and given vectors $v_i$ with $\sum v_i = 0$ we have $\sum y(v_i) = 0$. Let $(l,m)$ be a broken line, we claim that $E(\partial l / \partial t)$ is an increasing function on the linear components of $l$. At points where $l$ is non-linear by definition the tangent vector changes by a positive generator of the ray $\mathfrak{d}$, and $E$ is positive on such a generator. Similarly if $l$ crosses from one maximal cell to another then the convexity of $y$ shows that $y(\partial l / \partial t)$ increases.

  Now suppose that $(l_1, m_1)$ and $(l_2, m_2)$ are two broken lines from $v_1$ and $v_2$ respectively combining to form a pair of pants at $P$. By definition we have an equality:

  \[ y(\partial l_1/\partial t)|_{t=0} + y(\partial l_2/\partial t)|_{t=0} + y(P) = 0 \]

  Since $y(P)$ is non-negative and $-y(v_i)\le y((\partial l_i/\partial t)|_{t=0})$ we see that $y(v_1)+y(v_2)$ determines how many times such broken lines can bend. By restricting to the case where $F_0$ is an $I_4$ fibre we need only consider products with $y(v_i) = 1$, and so the only possible pairs of pants which can appear are the following:

  \begin{enumerate}

  \item Pairs where neither broken line bends.
  \item Pairs where one broken line bends along a monomial $m$ with $y(r(m)) = 1$.
  \item Pairs where one broken line bends along a monomial $m$ with $y(r(m)) = 2$.
  \item Pairs where both broken line bends along monomials $m_i$ with $y(r(m_i)) = 1$.
  \item Pairs where one broken line bends twice along monomials $m_i$ with $y(r(m_i)) = 1$.
    
  \end{enumerate}

Now $y(r(m))$ describes the order to which the corresponding curve meets the boundary, so we see that we must the calculate the number of curves tangent to the boundary at a single point to order one or two.
  
\end{example}

This also tells us the general form of the equations defining the mirror:

\begin{align}
  \vartheta_{D_1} \vartheta_{D_3} & = f_{(2,2)} \vartheta_{2D_2} + f_{(6,2)} \vartheta_{2D_4} + \sum f_{(i,1)} \vartheta_{D_i} + f_0   \label{eqn:MirrorToRES1} \\
  \vartheta_{D_2} \vartheta_{D_4} & = g_{(0,2)} \vartheta_{2D_1} + g_{(4,2)} \vartheta_{2D_3} + \sum g_{(i,1)} \vartheta_{D_i} + g_0 \label{eqn:MirrorToRES2}\\
  \vartheta_{D_i}^2 & = (1+r^{i}_{(i,2)}) \vartheta_{2D_i} + r^{i}_{(i+2,2)} \vartheta_{2D_{i+2}} + \sum r^{i}_{(i,1)} \vartheta_{D_i} + r^i_0\label{eqn:MirrorToRES3}
\end{align}
with $r^{i}_{(i,2)}$ and $r^{i}_{(i+2,2)}$ vanishing on the central fibre. The functions $f_{(2,2)}$, $f_{(6,2)}$, $g_{(0,2)}$, $g_{(4,2)}$, $r^{i}_{(i,2)}$ and $r^{i}_{(i+2,2)}$ correspond to pairs of broken lines not bending. The terms $f_{(i,1)}$, $g_{(i,1)}$ and $r_{i,1}^i$ come from pairs of broken lines where one of the lines bends once off of a ray with primitive $(k,1)$. The remaining terms correspond to pairs of broken lines bending in three different ways, either one broken line bends off of a ray with primitive $(k,2)$, both broken lines bends off of rays with primitives $(k,1)$ and $(k',1)$ or one of the broken lines bends twice, off of rays with primitives $(k,1)$ and $(k',1)$. Therefore from these we can extract two equations for the primitive $\vartheta$ functions and thus an embedding of the mirror family into affine space.

Whilst the original paper~\cite{MirrorSymmetryforLogCY1} used relative invariants we here use log Gromov-Witten invariants. The compatibility results of~\cite{ComparisonTheoremsForGromovWittenInvariantsOfSmoothPairsAndOfDegenerations} show that these two approaches agree. We recall the basic definitions of the subject.

  \begin{definition}

  A \emph{pre-log scheme} $X$ is a pair $(\underline{X}, \alpha_X: \mathcal{M}_X \rightarrow \mathcal{O}_{\underline{X}})$ of a scheme $\underline{X}$ (the \emph{underlying scheme}), a sheaf of monoids on $\underline{X}$ in the \'etale topology $\mathcal{M}_X$ and a homomorphism $\alpha_X$ of sheaves of monoids from $\mathcal{M}_X$ to the multiplicative monoid $\mathcal{O}_{\underline{X}}$. Saying that this is a \emph{log scheme} means that the restriction of $\alpha_X$ to the inverse image of $\mathcal{O}_{\underline{X}}^\times$ is an isomorphism.

  To any pre-log scheme there is an associated log structure given by taking the amalgamated sum

  \[ \mathcal{M}_X \oplus_{\alpha^{-1} \mathcal{O}_{\underline{X}}^{\times}} \mathcal{O}_{\underline{X}}^\times\]

  A morphism of log schemes $\phi : X \rightarrow Y$ is a pair $(\underline{\phi}, \phi^\#)$ with $\underline{\phi}: \underline{X} \rightarrow \underline{Y}$ a morphism of schemes and $\phi^\#$ fitting into a commutative diagram
\[    \begin{tikzcd}
      \underline{\phi} ^{-1} \mathcal{M}_Y \arrow{r}{\phi^\#} \arrow{d} & \mathcal{M}_X \arrow{d}\\
      \underline{\phi}^{-1} \mathcal{O}_{\underline{Y}} \arrow{r}{\phi^*} & \mathcal{O}_{\underline{X}}
    \end{tikzcd} \]

  Given a morphism of schemes $\underline{\phi}: \underline {X} \rightarrow \underline{Y}$ and the structure of a log scheme on $\underline{Y}$ the inverse image sheaf $\underline{\phi}^{-1} (\mathcal{M}_Y)$ on $X$ is naturally a pre-log structure. We call the associated log structure the \emph{pull-back log structure},
written $\phi^*\mathcal{M}_Y$. A morphism of log schemes $\phi: X \rightarrow Y$ is \emph{strict} if the natural
map $\phi^\#:\phi^*{\mathcal{M}_Y} \rightarrow \mathcal{M}_X$ is an isomorphism.

There is an interesting invariant of a log scheme which contains the combinatorial data not seen by classical geometry. The \emph{ghost} or \emph{characteristic sheaf} is the quotient $\mathcal{M}_X/\mathcal{O}_X^\times$.
  
\end{definition}

We introduced log structures to study pairs $(X,D)$, so we now explain how to construct a log structure from such a pair.

\begin{example}

  Let $\underline{X}$ be a smooth scheme and $\underline{D}$ be a simple normal crossings divisor. Write $\underline{U}$ for the complement of $\underline{D}$ in $\underline{X}$ and let $j$ be the inclusion of $\underline{U}$ into $\underline{X}$. Let $\mathcal{M}$ be the sheaf $j_* \mathcal{O}_{\underline{U}}^* \cap \mathcal{O}_{\underline{X}}$, and $\alpha_X$ the canonical inclusion. This defines a log scheme $X$, the \emph{divisorial log structure} on $\underline{X}$.

\end{example}

Importantly there is a moduli space of basic log stable maps to a target $X$ with a number of marked points and tangency data. The relative obstruction theory is given by the log cotangent bundle in the same way the usual obstruction theory is induced by the usual cotangent bundle.

\begin{definition}

  Let $f: X \rightarrow Y$ be a morphism of log schemes. A log derivation is a pair of maps $D, Dlog$ from $\mathcal{O}_{\underline{X}}$ or $\mathcal{M}_X$ respectively to a sheaf $\mathcal{E}$ satisfying:

  \begin{enumerate}
  \item $D$ is a derivation in the classical sense.
  \item $Dlog$ vanishes on sections of $f^{-1} (\mathcal{\mathcal{M}}_Y)$.
  \item For every element $m \in \mathcal{M}_X (U)$ one has an equality $D (\alpha_X(m)) = \alpha_X (m) Dlog(m)$.
  \end{enumerate}

\end{definition}

There is a universal sheaf classifying these, the logarithmic cotangent bundle $\Omega^{1, \dagger}_{X/Y}$. There is a natural pullback map, given $f:X \rightarrow Y$ and $g: Y \rightarrow Z$ there is a canonical morphism $f^* \Omega^{1, \dagger}_{Y/Z} \rightarrow \Omega^{1,\dagger}_{X/Z}$. We write $\omega^\dagger_{X/Y}$ for the top tensor power in the case where $\Omega^{1, \dagger}_{X/Y}$ is a vector bundle.

\section{Sections of a rational elliptic surface}

Bryan and Leung in~\cite{TheEnumerativeGeometryOfK3SurfacesAndModularForms} studied the Gromov-Witten theory of an elliptically fibred $K3$ surface with a section $E$ and fibre $F$. They considered genus $g$ stable maps with $g$ marked points in the class $E + nF$ and proved a formula for the Gromov-Witten invariants. For such a surface one cannot use the Gromov-Witten theory of $K$, but rather the reduced Gromov-Witten theory. This is defined to be the virtual class defined by $T^\circ$ in the following exact triangle
  \[ \tau_{\geq -1} R \pi _* \omega_{\pi} \otimes H^0 (K, \omega_K) \rightarrow R \pi_* f^* \Theta_{K/k} \rightarrow T^\circ \rightarrow\]
  In the paper Bryan and Leung constructed explicitly the moduli space of genus zero stable maps and reinterpreted it as the moduli space of maps on certain blow ups of $\PP^2$. In particular in section 5 they showed that any stable map in this moduli space is a union of a component mapping to $E$ and some combinatorial data describing a cover of different rational fibres. From this description it is clear that even infinitesimally there are no deformations smoothing the nodes between the component mapping to $E$ and those mapping to rational fibres.

  By explicit calculation they show that the two induced deformation theories on this moduli space of stable maps are quasi-isomorphic. Then in Lemma 5.7 the authors apply well known formulae for the Gromov-Witten theory of blow ups of $\PP^2$ to show that each component of the moduli space contributes either a one or a zero to the total count.

  They also performed the same analysis, deducing the same geometric results for the moduli space of stable maps to a rational elliptic surface $S$. In particular there is no infinitesimal smoothing of any of the nodes lying on the intersection of the component mapping to the section and those mapping to the fibres. In this case there was no need to work with reduced invariants either and in Theorem 6.2 they provide the following formula for $I_{0,0,E+nF}$, the number of unmarked rational curves in class $E+nF$:
\begin{equation}
\sum_{n=0}^\infty I_{0,0,E + mF} z^m = \prod _{m=1}^\infty (1- z^m)^{-12}
\end{equation}

But we also need to know that there are not too many classes of sections of a given degree on $S$. Once we combine these two results we will produce a bound on the total number of sections. Thankfully bounding the number of curves is an easy exercise in classical algebraic geometry:

\begin{lemma}[Bounds on the number of sections]

Let $d \in \mathbb{N}$, define the \emph{Goldilocks zone} 
\[
GZ(S, d) = \{ [C] \in A_1 (S) \mid [C] \sim dH - \sum a_i E_i, \:  [C].[F] = 1, \: p_a([C]) \geq 0 \}.
\]
\Lawrence{where $p_a([C])$ is the arithmetic genus of a generic member of this family.} Then $| GZ (S,d) | \leq (N\sqrt{d} + N)^{9}$ for some $N$ independent of $d$.

\label{thm:BoundedlyManySections}

\end{lemma}

\begin{proof}
Let $f: C \rightarrow S$ be a stable map. There is an inequality $g(C) \leq p_a (f_* [C])$ and so this counts the number of possible curve classes for stable maps of genus zero. To prove \Lawrence{this lemma} we apply the genus formula together with basic intersection theory. Let $[C] = d[H] - \sum a_i [E_i]$ be a curve class with $[C].[K_S] = -1$, so
$3d-\sum a_i=1$. Then the genus formula then states that $p_a(C) = 1 + 1/2([C].[C]-[C].[F])$ and
$p_g(C)\le p_a(C)$, so
\begin{align*}
0 \leq \; & 1 + 1/2([C].[C]-[C].[F])\\
\leq \; & 3/2 + 1/2 [C].[C]
\end{align*}
hence
\begin{align}
\begin{split}
0 \leq \; & 3 + [C].[C]\\
\leq \; & 3 + d^2 - \sum a_i^2
\end{split}
\label{inequality}
\end{align} 
By the arithmetic-geometric inequality $\sum a_i^2$ is minimised when all the $a_i$ are equal. We need to find a bound for $|a_i-a_j|$ in terms of $d$ given
the inequality above. Thus suppose
that $a_1 + k = 1/8 \sum_2^9 a_i$. We will bound $k$. Again $\sum a_i^2$ is minimised when $a_2 = a_3 = \cdots = a_9$. Using $3d = \sum a_i + 1$ we can write out everything explicitly
\begin{align*}
a_1 & = \frac{3d-1-8k}{9},\\
a_2 = a_3 = \ldots & = \frac{3d - 1 + k}{9}.
\end{align*}
Thus \Lawrence{\eqref{inequality}} implies
\begin{align*}
3 + d^2 & \geq \sum a_i^2\\
& \geq \left(\frac{3d - 1 -8k}{9}\right)^2 + 8 \left(\frac{3d - 1 + k}{9}\right)^2\\
& \geq d^2 - \frac{2}{3} d + \frac{8}{9} k^2 + \frac{1}{9},
\end{align*}
and hence
\[
\frac{26}{9} + \frac{2}{3}d  \geq \frac{8}{9}k^2
\]
or
\[
\frac{13}{4} + \frac{3d}{4}  \geq k^2
\]
Thus we see that for $|k| > \sqrt d + 2$ there can be no classes satisfying
the above inequality. By symmetry we may assume that $a_1$ is the smallest of the $a_i$. This then shows that none of the $a_i$ can be more than $\sqrt{d} + 2$ from $(3d-1)/9$, thus the total possible combinations is bounded by a bound of the above form. 

\end{proof}

In particular the number of rational curves of degree at most $d$ grows slower than $2^d$.

\section{The unravelled threefold}

The idea for this section is to make use of the fibration $\rho$. We will take a family of double covers of $\PP^1$ and construct a threefold $X$ fibred by $K3$ surfaces with known singularities. Each fibre will be a double cover of $S$ and so possesses an elliptic fibration. We will show that this induces a strong relation between the moduli spaces of stable maps to $S$ and the moduli space of stable maps to $X$. Let us begin though by thinking about the geometric content properties of the curves we wish to count.

Let $C$ be an irreducible rational curve on $S$ with $C.F = 2$. Restricting the fibration $\rho$ to $C$ we therefore obtain a double cover $\PP^1 \rightarrow \PP^1$. By the Riemann-Hurwitz formula this map must ramify at two points and up to a choice of involution these two points uniquely specify the double cover. This suggests that there should be a map from the moduli space of stable maps in the class $[C]$ to the Hilbert scheme of two points on $\PP^1$.

Let $[C]$ be an effective curve class in $A_1 (S, \mathbb{Z})$ with $[C].[F] = 2$. We will denote the moduli space of stable rational unmarked maps to $S$ with image in class $[C]$ by $\StableModuli_{0,0} (S, [C])$. This moduli space carries a universal curve, $\pi: \mathcal{C}_{0,0} (S,[C]) \rightarrow \StableModuli_{0,0} (S, [C])$, and a universal stable map, $f: \mathcal{C}_{0,0} (S,[C]) \rightarrow S$. Taking the composition $\rho\circ f$ we obtain a family of branched double covers of $\mathbb{P}^1$. Passing to the stabilisation we therefore obtain a morphism $\StableModuli_{0,0} (S, [C]) \rightarrow \MTwoOne$. The space $\StableModuli_{0,0} (\mathbb{P}^1, 2)$ meanwhile is naturally isomorphic to $[\Hilb / C_2]$. Here the $C_2$ action on $\Hilb$ is trivial, reflecting the involution on a double cover of $\mathbb{P}^1$. Any stable map in $\MTwoOne$ is a double cover of $\mathbb{P}^1$ by a rational curve hence ramified over two points. The comparison map to $\Hilb$ sends such a curve to the two ramified points. Therefore via composition we obtain a morphism $ram: \StableModuli_{0,0} (S, [C]) \rightarrow [\Hilb/C_2]$.

\begin{figure}[h]
  \centering
  \begin{tikzpicture}
    \draw (-4, 1) ellipse (1.5cm and 0.5cm);
    \draw (-6, 1.5) .. controls (-5.335,1) .. (-6, 0.5);
    \draw (-2, 1.5) .. controls (-2.665,1) .. (-2, 0.5);
    \draw (-3, 1.2) -- (-3, 2);
    \draw [->] (-4, 0) -- (-4, -0.5);
    \draw (-6, -1) -- (-2, -1);
    \filldraw (-5.5, -1) circle (2pt);
    \filldraw (-2.5, -1) circle (2pt);
    \draw (-5.5, -1.5) node {$P$};
    \draw [dotted] (-5.5, 1) -- (-5.5, -1);
    \draw [dotted] (-2.5, 1) -- (-2.5, -1);
    \draw (-2.5, -1.5) node {$Q$};
    \draw [->](-0.5, 0) -- (0.5, 0);
    \draw (0,0.5) node {$ram$};
    \draw (4, 0) circle (2);
    \draw (7.3, 0) node {$\Hilb$};
    \filldraw (3.5, 0.5) circle (2pt);
    \draw (3.5, 0.2) node {$\{P,Q\}$};
    \end{tikzpicture}
  
  \end{figure}

Our goal will be to understand the images of different irreducible components of the moduli space $\StableModuli_{0,0} (S, [C])$ under this map $ram$. To do this we wish to understand how the fibres of $\pi$ vary as double covers of $\PP^1$. This requires that we be able to identify which components contribute to the double cover for which we label the different components as follows. Let $f:C \rightarrow S$ be a stable map corresponding to a closed point of $\StableModuli_{0,0} (S, [C])$. There are three types of components of $C$: contracted components, components covering fibres of $\rho$ and components which cover the base $\mathbb{P}^1$ of the fibration. We will assign one of the two labels $\fib$ and $\tr$ to each component. Components covering fibres will be labelled with $\fib$ whilst components covering the base will be labelled $\tr$. Any contracted component which is part of a chain linking two components covering the base will be labelled $\tr$ and all the remaining contracted components will be labelled $\fib$. This divides $C$ into two types of components based on these labels, which we will denote by $C_{\tr}$ and $C_{\fib}$. The component $C_{\tr}$ consists either of a single component mapping as a double cover to the base $\mathbb{P}^1$ or as a chain of contracted components joining two components mapping isomorphically to the base or is the union of two disconnected components. Note that $C_{\fib}$ may very well be disconnected. We write $C_{\contr}$ for the curve obtained from $C_{\tr}$ by contracting out the contracted components, noting that if $C_{tr}$ is connected then $C_{contr}$ admits a map to $S$. The following lemma shows that in a fibre of $ram \circ \pi$ the component $C_{\tr}$ is rigid, only the component $C_{\fib}$ deforms.

\begin{theorem}[Fibres of $ram\circ \pi$ vary]
\label{thm:varying-fibres}
The component $f(C_{\tr})$ is locally constant as a subscheme of
$S$ along fibres of $ram\circ\pi$. Hence 
the image under $ram\circ \pi$ of a family of curves in $\StableModuli_{0,0} (S, [C])$ with $f(C_{\tr})$ varying as a subscheme of $S$ is positive dimensional.

\end{theorem}

\begin{proof}

Let $B \xleftarrow{\pi} C \xrightarrow{f} S$ be a family of genus zero unmarked stable maps contained inside a fibre of $ram$. First we treat the case where $C_{\tr}$ is irreducible.

Fix a choice of section $E$ for the rational elliptic surface which we take to be the identity for the group law. Choose an \'etale $U\rightarrow B$ with dense
image such that $(C|_U)_{\tr}$ is isomorphic to $U \times \mathbb{P}^1$. 
Thus we may identify fibres over each $u \in U$ as follows, using the assumption
that $C_{tr}$ is irreducible. After passing to another \'etale cover we may identify $(C|U)_{tr}$ with the pull-back of the universal curve over $\MTwoOne$. By assumption the image in $\MTwoOne$ is a point $P$. Therefore for each $u \in U$ we may identify the curve $C_{u,tr}$ with the same double cover of $\PP^1$.

\begin{figure}
\centering
\begin{tikzpicture}
\node (P1) at (0,2) {$\PP^1$};
\node (P2) at (0,-2) {$\PP^1$};
\node (P3) at (2,3) {$\PP^1$};
\node (P4) at (2,0) {$S$};
\node (P5) at (2,-3) {$\PP^1$};
\node (P6) at (-2,3) {$\PP^1$};
\node (P7) at (-2,0) {$S$};
\node (P8) at (-2,-3) {$\PP^1$};
\draw [->] (P1) -- (P2) node[midway, left] {$2:1$};
\draw [->] (P1) -- (P3) node[midway, rotate=27] {$\cong$};
\draw [->] (P3) -- (P4) node[midway, left] {$u$};
\draw [->] (P4) -- (P5) node[midway, left] {$\rho$};
\draw [->] (P1) -- (P6) node[midway, rotate=-27] {$\cong$};
\draw [->] (P6) -- (P7) node[midway, left] {$v$};
\draw [->] (P7) -- (P8) node[midway, left] {$\rho$};
\draw [->] (P2) -- (P5) node[midway, rotate=-27] {$\cong$};
\draw [->] (P2) -- (P8) node[midway, rotate=27] {$\cong$};
\end{tikzpicture}
\caption{}
\label{identificationoffibres}
\end{figure}

Fix a choice of closed point $u \in U$ and so a stable map to $S$ and for each $v \in U$ we identify the two covers via the above isomorphism as described in Figure~\ref{identificationoffibres} and consider $f(C_{tr,v}) - f(C_{tr, u})$. For $u = v$ this defines a double cover of the section $E$, while if $f(C_{\tr})$ is varying then for $u \neq v$ this is not a double cover of that section. By rigidity of sections this does not occur.

Now if $C_{\tr}$ is reducible, $C_{\contr} = C_1 \cup C_2$ then both $C_1$ and $C_2$ are sections and admit no deformations. In particular this shows that $f$ is constant on these fibres.

\end{proof}

Let us now study the deformation theory of the different components of $C$. The virtual dimension of the moduli space of genus $g$, $n$-marked stable maps to $S$ in a class $\beta$ is given by
\[(1-g)(dim \: S - 3) +n - K_S.\beta\]
applying this to our situation we see that $C$ moves in at least one dimensional family, and if $C_{tr}$ is connected then $C_{contr}$ also moves in a one dimensional family. We will work to relate deformations of these two curves. In any case there are three types of irreducible components of $\StableModuli_{0,0} (S, \beta)$, \emph{point components} whose image is dimension zero, \emph{curve components} whose image is dimension one and \emph{bubble components} which surject onto $\Hilb \cong \mathbb{P}^2$. Schematically the types of curves which can occur in each component is contained in Figure~\ref{fig:animals}. We can draw representative curves of each type in the class 
$E_1 + E_2 + 3F$ for sections $E_1, E_2$ satisfying $E_1=E_2$, $E_1\cdot E_2=1$
or $E_1\cdot E_2=0$ in the three cases depicted. 
The dashed line in the first figure shows that this is a double cover. We will later prove that these are representative of the general case.

\begin{figure}[h]
    \centering
    \begin{subfigure}[b]{0.3\textwidth}
        \begin{tikzpicture}
        \clip[draw](-2,-2) rectangle (2,2);
        \draw (-2,-0.3) -- (2, 0.5);
        \draw[dashed, ultra thick] (-2,-0.3) -- (2, 0.5);
        \NodalCurve{0.3}{0.7}{0.15}{0.4}
        \NodalCurve{-0.7}{-0.6}{-0.15}{0.4}
        \NodalCurve{1.2}{0.9}{0.15}{0.4}
      \end{tikzpicture}
      \caption{Bubble}
      \label{fig:gull}
    \end{subfigure}
    \begin{subfigure}[b]{0.3\textwidth}
      \begin{tikzpicture}
        \clip[draw](-2,-2) rectangle (2,2);
        \draw (-2, 0) .. controls (0,0) .. (-2,-2);
        \draw (2, 0) .. controls (-0,0) .. (2,2);
        \NodalCurve{0.65}{1}{0.15}{0.3}
        \NodalCurve{0.65}{-0.4}{0.15}{0.3}
        \NodalCurve{-1}{0.4}{-0.15}{0.3}
      \end{tikzpicture}
        \caption{Curve}
        \label{fig:tiger}
    \end{subfigure}
    \begin{subfigure}[b]{0.3\textwidth}
      \begin{tikzpicture}
        \clip[draw](-2,-2) rectangle (2,2);
        \draw (-2,1.4) -- (2, 1.2);
        \draw (-2,-1.4) -- (2, -1.2);
        \NodalCurve{0}{0}{1}{1}
        \NodalCurve{-0.6}{1}{-0.15}{0.3}
        \NodalCurve{-1.2}{-1}{-0.15}{0.3}
      \end{tikzpicture}
        \caption{Point}
        \label{fig:mouse}
    \end{subfigure}
    \caption{}
    \label{fig:animals}
\end{figure}

Let us begin by proving that if $C$ is in a point component then $C_{tr}$ is not connected.

\begin{lemma}[The structure curves over point components]

Let $\mathcal{M}$ be a point component. Then the restriction of the universal curve to this component decomposes as two constant components each mapping to potentially distinct sections joined by a tree of curves containing at least one component mapping to a fibre.

\end{lemma}

\begin{proof}

  If the transverse component $C_{\tr}$ is smooth and connected then \Lawrence{we have seen that we can deform it as} a double cover of $\mathbb{P}^1$. This leaves how to glue the component $C_{\fib}$ to this deformation. Suppose that $C_{\fib}$ meets $C_{\tr}$ in $P_1, \ldots , P_n$. \Lawrence{If none of these are ramification points of the map $C_{\tr} \rightarrow \PP^1$ then this deformation gives unique deformations of the $P_i$. If not then take the double cover $\AA^1 \rightarrow \AA^1$ given by $z \mapsto z^2$. Over this family there is a unique deformation given by choosing a branch to move along}. This allows us to lift the glueing data to the entire deformation contradicting Lemma \ref{thm:varying-fibres} and the definition of a point component. Therefore the universal curve must be generically reducible containing two distinct sections. The same argument applies if $C_{\tr}$ is the union of two sections meeting at a point.

Now suppose that $C_{\tr}$ is connected but not smooth hence of the form two sections joined by a chain of contracted components. Then by stability each contracted component in the chain must have an attached component mapping to a rational fibre.

\end{proof}

Now these point components are not relevant to the rest of our work, they cannot carry the correct log structure to be tangent to $F_0$ at a single point, nor are they generic enough to appear in our main construction. However the bubble components for sure will contribute, and controlling these will be key to our construction. In particular we can show that curve and bubble components do not meet.

\begin{theorem}[Curve and bubble components are disjoint]

Let $\mathcal{M}$ be a bubble component. Then $\mathcal{M}$ does not meet any curve components of the moduli space. 

\end{theorem}

\begin{proof}

We will begin by showing $f(C_{\tr})$ does not vary for any stable map in this space, a double cover of some section $\sigma$. We do this by first studying the restriction of the family to the boundary of the moduli space. Let $\partial \Hilb$ denote the conic of degenerate double covers inside $\mathcal{H}ilb^2 (\mathbb{P}^1) \cong \mathbb{P}^2/C_2$. Take a point $P \in \partial \Hilb$ such that the ramification point of the corresponding double cover $C_{\tr}\rightarrow \mathbb{P}^1$ is distinct from the images of any of the rational fibres of $S$. In this case, $C_{\tr}$ is reducible, the union of two sections meeting at a point which does not lie on one of the singular fibres and admits deformations which remain reducible. Sections however are rigid and so this can only happen if one has a double cover of a single section, $\sigma$.

Now suppose a curve component intersects $\mathcal{M}$. In this case, one
can find the spectrum of a DVR $B = \Spec R$ with closed point $b_{0}$ and generic point $b_{gen}$ and
a morphism $cl:B\rightarrow \StableModuli_{0,0}(S,[C])$ such that
$cl(b_0)\in \mathcal{M}$, $cl(b_{gen})\not\in\mathcal{M}$ and $cl(b_{gen})$
lies in a curve
component. Let
$\pi: C \rightarrow B$ be the pullback of the universal curve and $f: C \rightarrow S$ the corresponding stable 
map.

In fact, it is enough to show the following.
First, replace $C$ by removing the closures of irreducible components
of $C_{gen}$ mapping into the rational fibres of $S$ and taking the closure inside $C$. Thus we may assume
that $(C_{gen})_{\tr}=C_{gen}$. If we can show that $(C_0)_{\tr}=C_0$, then
as necessarily $f_*([C_0])=2[\sigma]$ and $\sigma^2=-1$, we see
that $f(C_{gen})\subseteq \sigma$ also. Thus $f$ yields a double cover
of $\sigma$, and hence $cl(b_{gen})$ in fact lies in a bubble component.

Our goal will be to decorate $C_{gen}$ with the structure of a log stable map. Then since the moduli space of log curves is proper the family $B$ admits a unique completion, and this completion must be a base-change of the classical stable map. But we will now see that results of~\cite{TheTropicalVertex} lead to
a contradiction. To this end we take the divisorial log structure on $S$ coming from the union 
$\cup F_i$ of all singular fibres.

To give a log structure on $C_{gen}$ we need to control the location of the singularities. Therefore we first exclude the case
that $C_{gen}$ has any nodes on the transverse component. Indeed, if $C_{gen}$ possesses a node then it is
reducible and the two components map to sections of $S\rightarrow \mathbb{P}^1$.
But then the family is constant since sections admit no deformations.
Thus we assume $f_{gen}^{-1}(\cup F_i)$ is a set of smooth points of $C_{gen}$
and we give $C_{gen}$ the divisorial log structure given by this divisor.
We then have a log morphism $C_{gen}\rightarrow S$.

Now we know that the generic fibre has a log structure and hence there is a unique choice of limit of stable log maps, possibly after passing to a branched
cover of $B$, see \cite{LogarithmicGromovWittenInvariants}, Theorem 4.1. We now claim that there are no log stable curves with components mapping into the fibre $F_0$. This is a consequence of the proof of Proposition 4.3 of~\cite{TheTropicalVertex}. This shows that $C_{gen}$ cannot degenerate to one of these curves. Thus we conclude all curve components are disjoint from $\mathcal{M}$.

\end{proof}

Now if we take a general point of $\Hilb$ there is a canonical way to produce a $K3$ surface by taking the fibre product over $\PP^1$. By the work of Bryan and Leung the moduli space of stable maps to an elliptically fibred $K3$ surface in a class $E + nF$ is independent of the choice of surface. Let us use this idea to describe the fibres of $\pi$ over a bubble component.

\begin{theorem}[A fibration of bubble components]

Let $\mathcal{M}$ be a bubble component. The map $\mathcal{M} \rightarrow \Hilb$ fibres $\mathcal{M}$ by moduli spaces of stable maps to (potentially singular) $K3$ surfaces. This fibration is trivial away from a codimension one set.
  
\end{theorem}

\begin{proof}

One might think that the fibration should be trivial as they correspond to different parametrisations of the rational fibres. This is only true generically but fortunately we can describe the locus where it does not hold. Let $R$ denote the locus inside $\Hilb$ of stable maps ramifying over the images of the $F_i$. Away from $R$ the fibration will be trivial since the only deformations are deformations of the rational tails.

Let $P \in \Hilb \setminus R \cup \partial \Hilb$ be a smooth double cover not ramifying over one of the rational fibres. Let $f: C \rightarrow S$ and $\pi: C \rightarrow B$ be a curve in $ram^{-1} P$ with $B$ an infinitesimal extension over $k[t]/t^2$. Let $B_0, C_0$ denote the reduced structure. We wish to show that this deformation of $C_0$ does not smooth any of the nodes between $C_{0,tr}$ and the other components. To do this we construct a moduli space of stable maps to a $K3$ surface where none of these smoothings exists and compare the obstruction theories. Take the following diagram
\[
\xymatrix@C=30pt
{
K\ar[r]^d\ar[d]&S\ar[d]^{\rho}\\
C_{0,tr}\ar[r]_{\rho\circ f}&\mathbb{P}^1
}
\]

This defines a smooth elliptically fibred $K3$ surface $K$ (here we are using the assumption that the curve does not ramify over the images of any of the rational fibres). $C_0$ maps to both $C_{0, tr}$ and $S$ and hence maps to $K$ as a section together with a collection of covers of the rational fibres. This shows that the fibres of $ram$ are isomorphic to moduli spaces of stable maps to a $K3$ surface. In particular by the beginning of section 5 ~\cite{TheEnumerativeGeometryOfK3SurfacesAndModularForms} there are no infinitesimal deformations smoothing the nodes lying on $C_{0,tr}$ since there are none on $K$. In particular this shows that over this set the fibration is trivial.

We now turn our attention to the boundary $\partial \Hilb$. We will show
that if $P\in\partial\Hilb$, then no map $\Spec k[\epsilon]/(\epsilon^2)
\rightarrow ram^{-1}(P)$ produces an infinitesimal deformation of stable
maps which
smooths any nodes between $C_{\tr}$ and $C_{\fib}$. Since $\partial\Hilb$ meets every codimension one locus in $\Hilb$ this would show that the set of points on which at least one of the deformations smooths these nodes is codimension at least two. Take $P \in \partial \Hilb$ and $f: C \rightarrow S$, $\pi:C \rightarrow B$ an infinitesimal deformation of a point $\pi: C_0 \rightarrow B_0$, $f_0: C_0 \rightarrow S$ in $ram^{-1}P$. Then $C_{0}$ contains a chain of rational curves connecting two components $D_1$ and $D_2$ which are mapped to sections of $S$. Since the deformation remains within the fibre it fails to smooth at least one node in this chain. Removing this node disconnects the domain curve $C_0$ into two components $C_a$ and $C_b$ and the deformation restricts to each of these. But $C_a$ and $C_b$ are sections together with rational tails, hence deformations of them fail to smooth any of the nodes connecting the rational tails to the section by the discussion of Section 5 of ~\cite{TheEnumerativeGeometryOfK3SurfacesAndModularForms}. Pulling back to the curve $C_0$ we obtain the desired statement.

Finally we claim that the fibres over any point in $\Hilb \setminus R$ are isomorphic. This is trivial from the above description. All these spaces are just moduli spaces of covers of the rational tails. It is worth saying that it is not true that this fibration is globally trivial. We can construct examples where the fibres of $ram$ over $R \cap \partial \Hilb$ are much higher dimension than the surrounding fibres since there could be more moduli in the way we connect the two transverse components together and such examples exist for three sheeted covers of one of the rational tails.

\end{proof}

Given an irreducible such curve, $C$ it is an easy calculation that the product $C \times_{\PP^1} S$ is a $K3$ surface. But $C$ maps both to itself via the identity and to $S$, compatibly over $\PP^1$. Therefore it also maps to this $K3$ surface. The moduli space we are considering provides us with a collection of components mapping to $\Hilb$. Of these we will be interested in those whose image is at least dimension one so we should search for a curve inside $\Hilb$. So let us perform this construction in a family setting.

\begin{construction}

   Let $D$ be a generic tri-degree $(1,1,2)$ surface inside $\mathbb{P}^1 \times \mathbb{P}^1 \times \mathbb{P}^1$ with projection maps $\pi_L$, $\pi _M$ and $\pi_R$ to the respective factors. We view this as a family of curves over $\mathbb{P}^1$ via the projection $\pi_L$. Taken this way each fibre is a bi-degree $(1,2)$ hypersurface inside $\mathbb{P}^1 \times\mathbb{P}^1$, so via $\pi_L$ the graph of a double cover of $\mathbb{P}^1$ by $\mathbb{P}^1$. This produces a morphism $\PP^1 \rightarrow \MTwoOne$ such that the family $\pi_L : D \rightarrow \PP^1$ is the pullback of the universal family of double covers. We wish to know that the trivial $C_2$ action on the Hilbert scheme $\Hilb$ is intertwined with the involution of $D$ swapping the two sheets of the double cover.

\end{construction}

  \begin{lemma}

    With the above notation the family $[D/C_2] \rightarrow [\PP^1/C_2]$ is the pull-back of the universal family of double covers of $\PP^1$ over the base.

  \end{lemma}

  \begin{proof}

    By construction the family $D$ is the pull-back of the universal family of
double covers to $\PP^1$. Therefore we have a diagram

    \[\begin{tikzpicture}
    \node (P1) at (0,0) {$\PP^1$};
    \node (M21) at (4,0) {$\MTwoOne$};
    \node (P1/2) at (2,1) {$[\PP^1/C_2]$};
    \node (D) at (0,2) {$D$};
    \node (C) at (4,2) {$\mathcal{C}$};
    \node (D/2) at (2,3) {$[D/C_2]$};
    \draw [->] (P1) -- (M21);
    \draw [->] (P1) -- (P1/2);
    \draw [->] (P1/2) -- (M21);
    \draw [->] (D) -- (P1);
    \draw [->] (C) -- (M21);
    \draw [->] (D/2) -- (P1/2);
    \draw [->] (D) -- (C);
    \draw [->] (D) -- (D/2);
    \draw [->] (D/2) -- (C);
    \end{tikzpicture}\]
    where the front face is a pullback square. Now $\PP^1$ is simply connected, hence there is only the trivial $C_2$ torsor over $\PP^1$. This shows that $\PP^1 \rightarrow [\PP^1/C_2]$ has a chart given by $\PP^1 \coprod \PP^1 \rightarrow \PP^1$. Taking the pull-back of $[D/C_2]$ to this and using that $D$ is the universal curve over $\PP^1$ we obtain the desired result.

\end{proof}

Let us explore some general features of this via an example.

\begin{example}

  Let us study the family with defining equation
  \[ x_1x_2x_3^2 + x_1y_2x_3y_3 + 2y_1x_2x_3y_3 + y_1y_2y_3^2 \]
where the $i^{th}$ factor has coordinates $x_i$ and $y_i$. The fibre over $(x_2, y_2)$ is the graph of a double cover of $\mathbb{P}^1$, ramifying where the discriminant $(x_1y_2+2y_1x_2)^2 - 4 x_1x_2y_1y_2 = x_1^2y_2^2 + 4y_1^2x_2^2$ vanishes. This gives two values for $(x_1:y_1)$ for each value of $(x_2:y_2)$. The fibre is singular just when these two values coincide, so when the discriminant of this quadratic vanishes, $-16 x_2^2y_2^2$. Therefore there are two singular fibres each with multiplicity two.
  
\end{example}

In general this shows that there will be four singular fibres and it is easy to construct examples where these occur as distinct fibres. We have constructed here a collection of double covers of $\mathbb{P}^1$, so the base $\mathbb{P}^1$ should map to $\Hilb$. The degenerate fibres inside $\Hilb$ form a conic inside $\mathbb{P}^2$, so by counting intersection points we know that the image of this family is also a conic. Therefore to search for curve and bubble components of the moduli space we are lead naturally to study the following fibre product:
\[\centerline{\xymatrix{X \ar[r]^{d} \ar[d] & S \ar[d]^{\rho} \\ D \ar[r]^{\pi_L} & \mathbb{P}^1}}\]
We call the constructed threefold $X$ an \emph{unravelled threefold}. Such a threefold is the intersection of a $(3,1,0,0)$ hypersurface and a $(0,1,1,2)$ hypersurface inside $\mathbb{P}^2 \times (\mathbb{P}^1)^3$. The $(0,1,1,2)$ hypersurface we may take to be generic whilst the $(3,1,0,0)$ hypersurface depends on the choice of rational elliptic surface $S$. The first hypersurface gives rise to a choice of line $\iota: [\PP^1/C_2] \rightarrow \MTwoOne \cong [\PP^2/C_2]$. We will restrict our range of choices later in the discussion to reflect the surfaces we choose to work with. Let $\mathcal{M}$ be a component of the moduli space $\StableModuli_{0,0} (S, f_*[C])$ and consider the image of $\mathcal{M}$ under the morphism $ram:\StableModuli_{0,0}(S,f_*[C])\rightarrow \mathcal{H}ilb^2 (\mathbb{P}^1) \cong \mathbb{P}^2/C_2$. Whenever the image intersects the image of the family of double covers of $\PP^1$ one can lift the corresponding stable map to $X$. The existence of a lift suggests that $\StableModuli_{0,n} (X, \tilde{f}_* ([C])) \cong \StableModuli_{0,n} (S, f_*[C]) \times_{\mathcal{H}ilb ^2 \mathbb{P}^1} \mathbb{P}^1$ along the map $ram$. Soon we will see that this description induces a relation between the Gromov-Witten invariants of these two moduli spaces.
  
\section{Calculating Gromov-Witten invariants}

We can describe this space $X$ as a nef complete intersection inside $\PP^2 \times \PP^1 \times \PP^1 \times \PP^1$. This allows us to apply techniques of Givental from~\cite{AMirrorTheoremforToricCompleteIntersection} to calculate the Gromov-Witten theory of $X$. We will begin by giving an explanation of this in the case of $S$ itself, and then explain why this is not enough to calculate the relative invariants we want. Instead we will prove that the $J$ function of $X$ is holomorphic, which provides strong bounds on the growth of the coefficients.

\begin{example}

  From the description above the threefold $X$ is the intersection of a $(3,1,0,0)$ hypersurface and a $(0,1,1,2)$ hypersurface in $\PP^2 \times \PP^1 \times \PP^1 \times \PP^1$.

\end{example}

Such hypersurfaces were studied in~\cite{AMirrorTheoremforToricCompleteIntersection}. The Gromov-Witten theory of a toric variety is very well understood and by using these line bundles we can localise the invariants to the intersection. Givental constructed solutions to two different quantum differential operators, and then proved that the solution spaces to these two were equal. Therefore there is a change of basis relating the constructed solutions. Let us recall the theory, starting with a toric variety $\mathbb{T}$ with codimension one strata $D_\rho$. Take a basis for the rational Chow ring $A^* (\mathbb{T})$, $\{H_i\}_1^n$, and a dual basis $\{H^i\}_1^n$ under cup product.

\begin{definition}

The $I_X$ and $J_X$ functions are defined for a smooth complete intersection $X \subset \mathbb{T}$. Suppose that $X$ is the common vanishing of sections of line bundles $\mathcal{L}_i$ with $-K_\mathbb{T} - \sum c_1(\mathcal{L}_i)$ nef. Then the $I_X$ function is defined by

\begin{align*}I_X (t_0, \ldots t_n) &= \\ e^{(t_0+\sum t_i H_i) / \hbar} Eul& (\oplus \mathcal{L}_i) \sum_\beta q^\beta \frac{\prod_i (\prod_{m=1} ^{\mathcal{L}_i (\beta)} (c_1 (\mathcal{L}_i) + m\hbar) \prod_\rho \prod_{m=-\infty}^{0} (D_\rho + m\hbar))}{\prod_\rho \prod_{m=-\infty}^{D_\rho \cap \beta} (D_\rho + m\hbar)} \end{align*}
where $q^{\beta} = e^{\sum t_i H_i \cap \beta}$ and the sum is over $\beta$ effective. This function is valued in $A^* (\mathbb{T})[\![t_0, \ldots t_n, \hbar^{-1}]\!]$. The function $J_X$ is a generating function for the gravitational correlators of $X$ and is valued in $A^* (X)[ \! [t_0, \ldots t_n, \hbar^{-1} ] \! ]$. It is defined by the formula
\[J_X (t_0 \ldots t_n) = e^{(t_0+\sum t_i H_i)/ \hbar} \left(1 + \sum_{\beta, i, k} \hbar^{-(k+1)} q^\beta \langle \tau_k H_a, 1 \rangle_{0,\beta} H^a \right)\]
where $\langle \tau_k H_a, 1 \rangle_{0,\beta}$ are the gravitational descendants, the sum is over positive $i, k$ and effective curve classes $\beta$.

\end{definition}

These two functions are not valued in the same ring so cannot be equal. Let $\iota: X \rightarrow \mathbb{T}$ be the inclusion. The main result of~\cite{AMirrorTheoremforToricCompleteIntersection} is an equality of generating functions $I_X(t_i) = \iota_* J_X (s_i)$  after some change of basis $t_i \mapsto s_i \in A^* (\mathbb{T})[\! [t_0, \ldots t_n ]\!]$.

\begin{theorem}[Mirror symmetry for Givental's $I$ and $J$ functions]

The functions $I_X$ and $\iota _* J_X$ are equal up to a homogeneous change of variables of the form $t_0 \mapsto t_0 + f_0 (z^{\beta}) \hbar + h (z^\beta)$, $t_i \mapsto t_i + f_i (z^\beta)$ where the $f_i$ and $h$ are homogeneous power series of weights $\deg f_i = 0$ and $\deg h = 1$ with degrees of
the variables being given by $c_1 (\mathbb{T}) - c_1 (\sum \mathcal{L}_i) = \sum deg (z^{H_i}) H_i$, $\deg \hbar = 1$, $\deg t_0=1$ and $\deg t_i=0$ for
$i>0$.

\begin{proof}

  See~\cite{AMirrorTheoremforToricCompleteIntersection} Theorem 0.1.
  
\end{proof}

\end{theorem}

A well written explanation of how to perform these calculations can be found in~\cite{MirrorSymmetryAndAlgebraicGeometry}. The rational elliptic surface $S$ is itself a $(3,1)$ hypersurface in $\PP^2 \times \PP^1$. Unfortunately this is not powerful enough to reconstruct the relative or open invariants and so it cannot answer our questions by directly calculating with $S$. Instead let us calculate the $J$ function of $X$ and attempt to interpret the answer geometrically. Since we are practising numerology we will not attempt to prove our claims. Firstly we take a basis for the cohomology $H_1, \ldots, H_4$ given by the duals of pullbacks of hyperplanes from each factor.

The invariants we wish to calculate intersect $H_2$, the class which restricts on each $K3$ fibre to a fibre of the fibration of that $K3$, at a single point. Therefore the number of such curves appear as the coefficient of $H_2$ in the corresponding term in the $J_X$ function. We must double the invariant to account for the tangency condition and double it again since we are really intersecting with a conic inside $\Hilb$. This suggests that each curve on $S$ should be counted with multiplicity four compared to the corresponding curve on $X$.

When we expand out the function $J_X$ the first three coefficients we calculate, the coefficients of $z^{(0,2,0,1)}, z^{(1,2,0,1)}$ and $z^{(2,2,0,1)}$, are $-9, 144$ and $1980$. The first of these we can easily explain, there are nine sections each admitting a bubble component, by the double cover formula the relative invariant should be $-9/4$. The 144 reflect the fact that for each family class $H - E_i$ there are four lines which are tangent to the boundary to order two, the coefficient therefore is $9 \times 4 \times 2 \times 2 = 144$.

For the term 1980 we begin to see an interaction between the bubble components and the curve components. Each of the lines $H - E_i - E_j$ admits a bubble component, there are 36 choices of such classes so we expect a contribution of $-36 = -36/4 \times 2 \times 2$. The classes $2H - E_i - E_j - E_k - E_l$ give a curve component of which four elements are tangent to the boundary. This gives a contribution of $126 \times 4 \times 2 \times 2 = 2016$, combining these we obtain the predicted contribution of $1980$. Unfortunately as the degree grows the presumed correspondence with the log invariants becomes intractable. To find the next term we would have to answer the question ``given a rational cubic $E$ passing through seven points $P_i$ what is the virtual contribution from rational cubics pass through those points meeting $E$ in one other point of tangency order two?''. This is why we must pass to the threefold.

Our main result for this section is that $J_X$ is actually a holomorphic function. We believe that this should be part of the general yoga of hypergeometric equations, but do not know of a reference. Let us recall the definition and some key facts about holomorphic functions. We will prove that $I_X$ is holomorphic and then deduce the holomorphicity of $J_x$.

Given a vector $\underline{v} = (v_1, \ldots v_n) \in \NN^n$ we write $x^{\underline{v}}$ for the monomial $\prod x_i^{v_i}$. Recall that a holomorphic function in variables $x_1, \ldots x_n$ is given locally by an absolutely convergent power series $\sum a_{\underline{v}} x^{v}$ converging on some disc and that if such a power series exists then operations such as differentiation and so on may be performed term by term. Let $diag (\underline{v})$ be the function sending $v$ to $\sum (v_i + 1)$. Suppose now that $a_{\underline{v}}$ are real numbers satisfying $|a_{underline{v}}| < c r^{diag (\underline{v})}$ for some fixed choice of real constants $c$ and $r$. Then by construction the power series
\[ f (x_1, \ldots, x_n) = \sum a_{\underline{v}} x^{\underline{v}}\]
converges on the polydisc around the origin in $\CC^n$ of radius $1/r$. We call such a family of $a_{\underline{v}}$ \emph{exponentially bounded}. We provide a converse to this statement.

\begin{proposition}

  Let $f: D^n (0,r) \rightarrow \CC$ be a holomorphic function on a polydisc around $0$ of radius $r$. Then $f (x_1, \ldots x_n) = \sum a_{\underline{v}} x^{\underline{v}}$ with the $a_{\underline{a}}$ exponentially bounded.
  
\end{proposition}

\begin{proof}

  The existence of the power series expansion is standard. We apply the Cauchy Integral formula for many variables to find $a_{\underline{v}}$.
  \[a_{\underline{v}} = \frac{1}{2\pi i}^n\int_{\partial D^n(0,r/2)} \frac{f(x_1, \ldots x_n)}{x^{\underline{v}} x_1 \ldots x_n} dx_1 \ldots dx_n.\]
  Therefore we have
  \[|a_{\underline{v}}| \leq \frac{\max_{z \in D^n (0,r/2)}|f(x_1, \ldots x_n)|}{(r/2)^{diag(\underline{v})}} \]
  where $f(x_1, \ldots x_n)$ is bounded on $D^n (0,r/2)$ . Therefore we have the result.
\end{proof}

\begin{definition}

  We work in slightly greater generality than is standard. Let $R = \CC[t_1, \ldots t_m] / \langle t_1^{k_1}, \ldots t_m^{k_m} \rangle$ be a finite dimensional $\CC$ algebra. We say that an element of $R[\! x_1, \ldots x_n \! ]$ is holomorphic if the coefficients of each monomial in the $t_{i}$'s are all holomorphic in the $x_i$.
  \label{def:exp-bounded}
\end{definition}
  
  The class of holomorphic functions is closed under many operations. This is a list of properties we need of them.

\begin{lemma}[A holomorphic toolkit]

The set of germs of holomorphic functions around $0$ is closed under the following operations:

\begin{enumerate}
  \item If an element $f \in \mathbb{C} [ \! [ t_1, \ldots t_n, x_1, \ldots x_m ] \! ]$  is holomorphic then its image in the quotient ring $R [ \! [ x_1, \ldots x_m ] \! ]$ obtained by dividing out by the ideal $\langle t_1^{m_1},\ldots,t_n^{m_n}\rangle$ is holomorphic.
\item If $f = \sum a_{\underline{i}, \underline{j}} t^{\underline{i}} x^{\underline{j}}$ is holomorphic and $g = \sum b_{\underline{i}, \underline{j}} t^{\underline{i}} x^{\underline{j}}$ is such that $| a_{\underline{i}, \underline{j}} | \geq | b_{\underline{i}, \underline{j}} |$ then $g$ is holomorphic.
\item Taking products of elements.
\item Taking inverses of invertible elements.
\item Taking exponentials of elements.
\item Substitutions of the form $x_i \mapsto x_i f_i$ where the $f_i$ are also holomorphic around $0$.
\item If $s_i = x_i f_i (x_0,  \ldots , x_m)$ is a change of basis with $f_i$ holomorphic and non-vanishing at zero then this change of basis locally has an inverse where the inverse map is given by $x_i = s_i g_i (s_0, \ldots, s_m)$ with the $g_i$ holomorphic.
\end{enumerate}

\end{lemma}

\begin{proof}

  The proof of i) is trivial.
  
  The proofs of ii) through vi) are standard for genuinely holomorphic functions of multiple variables. However the first result allows us to lift holomorphic elements of $R [ \! [ x_1, \ldots x_m ] \! ]$ to genuinely holomorphic functions of more variables by treating the $t_i$ as being free under the natural inclusion and then take the quotient.

  The second result says that if a collection $a_{\underline{v}}$ is exponentially bounded then any subcollection of the $a_{\underline{v}}$ are exponentially bounded too.
  
  To see the final result note that the derivative $\partial s_i / \partial x_j = \delta_{ij}$. Therefore the change of basis $\CC^n \rightarrow \CC^n$ is locally invertible by the inverse function theorem with holomorphic inverses. Since this maps $(0, \ldots 0)$ to $(0, \ldots 0)$ it has the required form.
  
\end{proof}

We now show that $I_X$ is exponentially bounded. Then the function $J_X$ was constructed using the tools coming from the above toolkit, and hence is also holomorphic.

\begin{theorem}[The $I$ function is holomorphic]

The function $I_X$ is holomorphic for $X$ an unravelled threefold,
where the coordinates $x_i$ in Definition \ref{def:exp-bounded} coincide
with $e^{t_i}$ in $I_X$, and the coordinates $t_i$ in
Definition \ref{def:exp-bounded} coincide with the $H_i$ in $I_X$.

\end{theorem}

\begin{proof}

To do this we prove the the coefficients of $I_X$ form an exponentially bounded set. We write out the definition of the $I_X$ function

\begin{equation}
I_X (t_0, \ldots t_n) = e^{\sum t_i H_i / \hbar} Eul (\oplus \mathcal{L}_i) \sum_\beta q^\beta \frac{\prod_i (\prod_{m=1} ^{\mathcal{L}_i (\beta)} (c_1 (\mathcal{L}_i) + m\hbar) \prod_\rho \prod_{m=-\infty}^{0} (D_\rho + m\hbar))}{\prod_\rho \prod_{m=-\infty}^{D_\rho \cap \beta} (D_\rho + m\hbar)}
\end{equation}

To prove boundedness, with $x_i=e^{t_i}$ and $x^{(a_1,\ldots,a_4)}
=\prod_i x_i^{a_i}$, we only need to prove boundedness of the sum

\begin{equation}
\sum x^{(a,b,c,d)} \frac{\prod_{m=1} ^{3a + b} (3H_1 + H_2 + m\hbar) \prod_{m=1} ^{b+c+2d} (H_2 + H_3 + 2H_4 + m\hbar)}{\prod_{m=1}^a (H_1 + m\hbar)^3\prod_{m=1}^b (H_2 + m\hbar)^2\prod_{m=1}^c (H_3 + m\hbar)^2\prod_{m=1}^d (H_4 + m\hbar)^2}
\end{equation}
Fix a coefficient $H_{\underline{i}} = \prod H_i ^{\alpha_i}$ and consider the coefficient of $H_{\underline{i}} x^{(a,b,c,d)}\hbar^{-c + \sum \alpha_i}$. Expanding out the above product we obtain the expression
\begin{align*}
  \sum x^{(a,b,c,d)} \hbar^{-c} & \frac{(3a + b)! (b + c + 2d)!}{a! ^3 b!^2 c!^2 d!^2} \\
  & \frac{\prod_{m=1} ^{3a + b} (\frac{3H_1 + H_2}{m\hbar} + 1) \prod_{m=1} ^{b+c+2d} (\frac{H_2 + H_3 + 2H_4}{m\hbar} + 1)}{\prod_{m=1}^a (\frac{H_1}{m\hbar} + 1)^3\prod_{m=1}^b (\frac{H_2}{m\hbar} + 1)^2\prod_{m=1}^c (\frac{H_3}{m\hbar} + 1)^2\prod_{m=1}^d (\frac{H_4}{m\hbar} + 1)^2} \end{align*}
any monomial appearing in this consists of two parts, a term $\frac{(3a + b)! (b + c + 2d)!}{a! ^3 b!^2 c!^2 d!^2}$ with no $\hbar$ term times a correction factor involving a sum of products of terms $1/i$. The correction factor is at most a polynomial in $a,b,c,d$, hence exponentially bounded since there are only finitely many choices of $\underline{i}$. We now bound the term $\frac{(3a + b)! (b + c + 2d)!}{a! ^3 b!^2 c!^2 d!^2}$ using Stirling's formula:
\[
\frac{(3a + b)! (b + c + 2d)!}{a! ^3 b!^2 c!^2 d!^2} 
< k e^c \cdot \frac{(3a + b)^{3a+b} (b+c+2d)^{b+c+2d}}{a^{3a} b^{2b} c^{2c} d^{2d}} \cdot \frac{2 \pi \sqrt{(3a + b) (b + c + 2d)}}{\sqrt{512 \pi ^ 9 a^3 b^2 c^2 d^2 }}
\]
The middle term is the only one not obviously exponentially bounded. We rewrite it as

\begin{equation}
\frac{(3a + b)^{3a}}{a^{3a}} \cdot \frac{(3a+b)^b}{b^b} \cdot \frac{(b + c + 2d)^b}{b^{b}} \cdot \frac{(b + c+ 2d)^{c}}{c^{2c}} \cdot \frac{(b + c +2d) ^{2d}}{d^{2d}}
\end{equation}

Rearranging the first term we get $(3+b/a)^{3a}$ and we wish to compare this to $r^{a + b}$. Let $r = 27$, then $27^{a+b} = 3^{3a+3b}$ and we can rewrite the quotient $(3+b/a)^{3a} / r^{a + b}$ as
\[ \frac{(1+\frac{b}{3a})^{3a}}{3^b}.\]
Taking log and using that $\log (1+x) < x$ we find that
\[ 3a \log \left(1 + \frac{b}{3a} \right) - b \log (3) < b - b \log (3) < 0\]
Hence for $r > 27$ we see that $(3 + b/a)^{3a} < r^{a + b}$. There are similar choices of $r$ for each other term. Our choice of $r$  is the product of all of these minimal choices. Then since $r^{a+b+c+d}$ dominates all terms of the form $r^{a+b}$ etc. we see that the coefficients of $I_X$ are exponentially bounded. Thus $I_X$ defines a holomorphic function in a neighbourhood of the origin.

\end{proof}

\begin{corollary}

  The function $J_X$ is also holomorphic.
  
\end{corollary}

\section{Relating Gromov-Witten invariants}

This is the real technical heart of the paper. In this section we will prove that the Gromov-Witten count of curves on $X$ and on $S$ are related by a Gysin map induced by a regular embedding $\PP^1 \hookrightarrow \Hilb$. The invariants we actually wish to count are then related by the composition of two Gysin maps, one induced by an inclusion of components and the other by a regular embedding $\PP^1 \hookrightarrow \Hilb$. This inclusion of components causes some problems for us, since we do not know if we have thrown away components with large virtual degree.

Let us begin by proving our philosophy, that the moduli spaces of stable maps to $X$ and to $S$ are highly related.

\begin{lemma}[The unravelled threefold lifts curves]
Let $X \subset \mathbb{P}^2 \times (\mathbb{P}^1)^3$ be an unravelled three-fold, $\pi_M$ the projection to the third factor, a copy of $\mathbb{P}^1$,
and let
\[
d^{-1} \beta = \{ \alpha \in A_1 (X) \mid \pi _{M, *} \alpha = 0, d_* \alpha = \beta\}
\]
be the set of effective classes on $X$ contained inside a fibre of $\pi_M$  
and such that $d_* \alpha = \beta$. \Lawrence{We write $\StableModuli(X,d^{-1}(\beta))$ for the union of moduli spaces $\bigcup _{\alpha \in d^{-1} \beta} \StableModuli_{0,0} (X, \alpha)$}. Then consider the following diagram
\[
\xymatrix{\StableModuli(X,d^{-1}(\beta)) \ar@/^2pc/[rrrd]^{d_*} \ar@/_2pc/[rdd]^{ pr _X} \ar[rd]^{\psi} & & & \\
 & \mathbb{P}^1 \times_{\MTwoOne} \StableModuli_{0,0} (S, \beta)  \ar[rr]^{\iota} \ar[d] ^{pr _S } & & \StableModuli_{0,0} (S, \beta) \ar[d]^{ram} \\
 & \mathbb{P}^1 \ar[rr]^{\iota} & & \MTwoOne}
\]
where $pr_X$ associates to a stable map $f:C/B\rightarrow X$ the morphism
$B\rightarrow \mathbb{P}^1$ given by the ramification of the composed
map $\rho\circ d\circ f$, where $\rho:S\rightarrow \mathbb{P}^1$ is the
elliptic fibration.
Then the map $\psi$ is an isomorphism, and thus a two to one cover of $[\mathbb{P}^1/C_2] \times_{\MTwoOne} \StableModuli_{0,0} (S, \beta) \subset \StableModuli_{0,0} (S, \beta)$.
\end{lemma}

\begin{proof}
By genericity we may assume that the choice of map $\mathbb{P}^1 \rightarrow \Hilb$ avoids all point components and is transverse to the images of all curve components. We use the universal property of products to construct an inverse.

Let $S \xleftarrow{f} C \xrightarrow{\pi} B$ be a family of stable maps inside $\mathbb{P}^1 \times_{\MTwoOne} \StableModuli_{0,0} (S, \beta)$. Since $D$ is the restriction of the universal curve on $\MTwoOne$ to $\PP^1$ there is an induced morphism $C \rightarrow D$, and hence by properties of fibre products a unique morphism $C \rightarrow X$. This constructs an inverse to $\psi$, which we call it $\phi$.

We now need to prove that the two compositions $\phi \psi$ and $\psi \phi$ are in fact the identity maps. To show that $\psi \cdot \phi$ is the identity we show that $\iota \cdot \psi	 \cdot \phi = \iota$ and $pr_S \cdot \psi \cdot \phi = pr_S$. These are both automatic from the construction of $\phi$. Indeed, $\iota \cdot \psi \cdot \phi = d_* \cdot \phi = \iota$ and $pr_S \cdot \psi \cdot \phi = pr_X \cdot \phi = pr_S$.

Finally we show the composition $\phi \cdot \psi$ is the identity. Let $f: C \rightarrow X$ be a stable map in $\bigcup _{\alpha \in d^{-1} \beta} \bar{\mathcal{M}}_{0,0} (X, \alpha)$. The image $\psi (f:C \rightarrow X)$ is given by pushforward to $S$ and has a unique lift. But $f: C \rightarrow X$ is such a lift and so this composition too is the identity map.

\end{proof}

Since any degree two map from $[\mathbb{P}^1/C_2]$ to $[\mathbb{P}^2/C_2]$ is a regular embedding
there exists a Gysin map for the inclusion $\iota: \StableModuli (X, d^{-1}\beta) \rightarrow \StableModuli (S, \beta)$ covering $\mathbb{P}^1 \rightarrow \Hilb \cong [\mathbb{P}^2/C_2]$. We follow~\cite{VirtualPushforward} in constructing a compatibility datum between the two spaces.

\begin{theorem}[Compatibility of Gromov-Witten invariants betwen $S$ and $X$]

  For a generic choice of unravelling $d: X \rightarrow S$ there is an equality between the virtual classes associated to the obstruction theories on $\StableModuli_{0,0}(X, d^{-1} \beta)$ and $\StableModuli_{0,0} (S, \beta)$ under the Gysin map $\iota ^!$:
\[[\StableModuli_{0,0}(X, d^{-1} \beta)]^{vir} = 2 \iota^! [\StableModuli_{0,0} (S, \beta)] ^{vir}\]

\label{thm:SandXcomparisons}

\end{theorem}

\begin{proof}

We choose a generic unravelling which does not meet any point components and is transverse to any curve components. We have the following commutative diagram
\[
\xymatrix{X \ar[r]^{d} & S \\
\mathcal{C}_{0,0} (X,d^{-1} \beta) \ar [r]^{\iota} \ar[d]^{\pi_X} \ar [u]^{f_X} & \mathcal{C}_{0,0} (S, \beta) \ar[d]^{\pi_S} \ar [u]^{f_S} \\
\StableModuli_{0,0} (X,d^{-1} \beta) \ar [r]^{\iota} \ar[d]^{ram}  & \StableModuli_{0,0} (S, \beta) \ar[d]^{ram}\\
\mathbb{P}^1 \ar [r]^{\iota}& \MTwoOne
}
\]

The desired compatibility datum is a morphism of distinguished triangles
\[
\xymatrix{\cdots \ar[r] & (\iota^* R \pi_{S, *} f_S^* \Theta_{S/k})^{\vee} \ar[r]^{\phi} \ar[d]^{\gamma_1} & (R \pi_{X, *} f_X^* \Theta_{X/k})^{\vee} \ar[r] \ar[d]^{\gamma_2} & ram^* L_{\mathbb{P}^1 / \mathbb{P}^2} \ar[r] \ar[d]^{\gamma_3} & \cdots\\
\cdots\ar[r] & \iota^* L_{\StableModuli_{0,0}(S, \beta)/\mathfrak{M}} \ar[r] & L_{\StableModuli_{0,0}(X, d^{-1} \beta) / \mathfrak{M}} \ar[r] & L_{\StableModuli_{0,0}(X, d^{-1} \beta) / \StableModuli_{0,0}(S, \beta)} \ar[r] & \cdots}
\]

To prove this we must show that the cone over $(\iota^* R \pi_{S, *} f_S^* \Theta_{S/k})^{\vee} \rightarrow (R \pi_{X, *} f_X^* \Theta_{X/k})^{\vee}$ is quasi-isomorphic to $ram^* L_{\mathbb{P}^1 / \mathbb{P}^2}$. First let us show that the map $\gamma_3 : ram ^* L_{\mathbb{P}^1/\Hilb} \rightarrow L_{\StableModuli_{0,0}(X, d^{-1} \beta) / \StableModuli_{0,0}(S, \beta)}$ is an isomorphism. The morphism $\mathbb{P}^1 \rightarrow \Hilb$ is a local complete intersection morphism, defined by a single equation $F$ of degree two in $k[X,Y,Z]$. So in particular, $F$ forms a length one regular sequence
in $k[X,Y,Z]$. By genericity we may take the images of $\mathbb{P}^1$ and the curve components of $\StableModuli_{0,0}(S, \beta)$ to be transverse inside $\Hilb$. Therefore the sequence $F$ remains regular when pulled back to $\StableModuli_{0,0}(S, \beta)$ and $\StableModuli_{0,0}(X, d^{-1} \beta)$ is just the vanishing locus of the pullback of $F$. In the case of a regular embedding
there is an explicit formula for the cotangent complex, namely it is $\mathcal{I}/ \mathcal{I}^2[1]$ where $\mathcal{I}$ is the defining sheaf of ideals. Both source and target of $\gamma_3$ is $ram^*(\langle F \rangle / \langle F \rangle^2[1])$ and $\gamma_3$ is the identity.

Now taking the cone \Lawrence{over $\phi$} there is a morphism of triangles of the following form
\[
\centering
\xymatrix{... \ar[r] &(\iota^* R \pi_{S, *} f_S^* \Theta_{S/k})^{\vee} \ar[r] \ar[d] &(R \pi_{X, *} f_X^* \Theta_{X/k})^{\vee} \ar[r] \ar[d] & \Cone \ar[r] \ar[d]^{\delta} & ...\\
  ... \ar[r] & \iota^* L_{\StableModuli_{0,0}(S, \beta)/\mathfrak{M}} \ar[r] & L_{\StableModuli_{0,0}(X, d^{-1} \beta) / \mathfrak{M}} \ar[r] & L_{\StableModuli_{0,0}(X, d^{-1} \beta) / \StableModuli_{0,0}(S, \beta)} \ar[r] & ...}
\]
\begin{equation}
\label{diag:bigone}
\end{equation}
To begin with we need a result allowing us to commute pull-back and push-forward. We claim that there is an isomorphism
\[\iota^* R \pi _{S \; *} f_S^* \Theta_{S/k} \cong R \pi_{X \; *} \iota ^* f_S^* \Theta_{S/k} \cong R \pi_{X \; *} f_X^* d^* \Theta_{S/k}\]
The second isomorphism is automatic, so we focus on the first. By~\cite{EGA4}, III.2, Corollary 6.9.9 and the fact that families of curves are flat, cohomology and base change commute. Applying~\cite{GromovWittenInvariantsInAlgebraicGeometry}, Proposition 5, $f^* _S \Theta_{S/k}$ admits a resolution by a two term complex of vector bundles $b: F^{-1} \rightarrow F^0$ such that $\pi_* F^i = 0$ and $R^1 \pi_* F^i$ is a vector bundle. 
Now $\pi_* f^* \Theta_{S/k}$ is just the cokernel of $R^1 \pi _* (F^0) \rightarrow R^1 \pi _* (F^1)$. But we have seen that cohomology and base change commute, thus we have an isomorphism $\iota^* R^1 \pi _* (F^i) \cong R^1 \pi_* (\iota ^* F^i)$ and $\iota^* R^0 \pi_* (F^i) = 0$. Now this is enough to show that there is the desired equality.

This allows us to calculate $\mathcal{C}one$ from the distinguished triangle on $X$
\begin{equation}
\label{exact-sequence-upstairs} \rightarrow \Theta _{X/k} \rightarrow d^* \Theta_{S/k} \rightarrow \MacBeth \rightarrow 
\end{equation}
The complex $\MacBeth$ is a line bundle in degree minus one and supported only on the critical locus of $X$ over $S$ in degree zero. The fibres of the map $d: X \rightarrow S$ are isomorphic to the fibres of $d:D \rightarrow \PP^1$. By an elementary calculation this is only singular over two points of $\PP^1$ where the corresponding double cover degenerates. The choice of $D$ was arbitrary so we can arrange that these two singular points do not occur over the critical values of the map $S \rightarrow \PP^1$, so over the images of the singular fibres. This ensures that $H^0 (\MacBeth)$ is supported only on two elliptic fibres of the fibration $\rho: S \rightarrow \PP^1$.

Let us describe the pullback of \Lawrence{$\mathcal{D}$} to fibres of the universal curve of stable maps. Let $f: C \rightarrow X$ be a stable map. By necessity such a stable map factorises through a (possibly singular) $K3$ double cover of $S$. This shows that $L^0 f^* \MacBeth$ is supported only on a finite set since the universal curve can have no components mapping into an elliptic fibre. By applying the Grothendieck spectral sequence we therefore see that both $H^1 (R\pi_*Lf^*\MacBeth)$ and $H^{-1}((R\pi_*Lf^*\MacBeth)^\vee)$ vanish.

By pushing forward \eqref{exact-sequence-upstairs} and dualizing in the
derived category, we obtain a long exact sequence
\begin{align*}
0\rightarrow & H^{-1} ((R\pi_*f^* d^* \Theta _{S/k})^\vee)
 \rightarrow  H^{-1}((R\pi_* f^* \Theta_{X/k})^\vee)
 \rightarrow   H^{0}((R\pi_*Lf^*\MacBeth)^\vee)\\
 \rightarrow & H^{0}((R\pi_*f^* d^* \Theta _{S/k})^\vee) \rightarrow H^{0}((R\pi_* f^* \Theta_{X/k})^\vee)
 \rightarrow  H^{1}((R\pi_*Lf^*\MacBeth)^\vee) \rightarrow 0
\end{align*}
we see that $H^{1}((R\pi_*Lf^*\MacBeth)^\vee)$ vanishes by mapping via the canonical isomorphisms to the non-negative part of the cotangent exact sequence:
\[  \iota^* L_{\StableModuli_{0,0}(S, \beta)/\mathfrak{M}} \rightarrow L_{\StableModuli_{0,0}(X, d^{-1} \beta) / \mathfrak{M}} \rightarrow L_{\StableModuli_{0,0}(X, d^{-1} \beta) / \StableModuli_{0,0}(S, \beta) } \rightarrow \]
and applying the 5-lemma. We now have an exact sequence
\[0\rightarrow H^{-1} ((R\pi_*f^*d^* \Theta _{S/k})^\vee) \rightarrow H^{-1}((R\pi_* f^* \Theta_{X/k})^\vee) \rightarrow H^{0}((R\pi_*Lf^*\MacBeth)^\vee) \rightarrow \]
\[ \rightarrow H^{0}((R\pi_*f^* d^* \Theta _{S/k})^\vee) \rightarrow H^{0}((R\pi_* f^* \Theta_{X/k})^\vee) \rightarrow 0\]
and by the 5-lemma $H^{0}((R\pi_*Lf^*\MacBeth)^\vee)$ maps surjectively to $ram^* L_{\PP^1/\PP^2}$. 
Taking stalks at a generic point $\eta=\Spec K$ of the moduli space,
we may apply the Riemann-Roch theorem to the two tangent sheaf terms and using additivity of Euler characteristics we find that $H^{0}((R\pi_*Lf^*\MacBeth)^\vee)_{\eta}$ has dimension one as a vector space over $K$. Now $(R\pi_*Lf^*\MacBeth)^\vee$ has cohomology supported in degrees minus one and zero. Let $E^\bullet$ be a resolution of $R\pi_*Lf^*\MacBeth$, which is quasi-isomorphic to the truncation $\tau_{\leq 0} E^\bullet$ defined in~\cite{StacksProject} Tag 0118. Taking a projective Cartan Eilenberg double resolution of this we obtain a resolution of $R\pi_*Lf^*\MacBeth$ by projectives concentrated in non-positive degree. Therefore after dualising $R\pi_*Lf^*\MacBeth^\vee$ has a resolution concentrated in non-negative degree. This shows that $H^0(R\pi_*Lf^*\MacBeth^\vee)$ embeds into a vector bundle, and hence is torsion free. \Lawrence{We prove an analogous result to the statement that a surjective map of line bundles is an isomorphism}. Take the morphism
\[\delta: H^{0}((R\pi_*Lf^*\MacBeth)^\vee) \rightarrow L_{\StableModuli_{0,0}(X, d^{-1} \beta) / \StableModuli_{0,0}(S, \beta)} \]
of the diagram \eqref{diag:bigone}.
After localising at a point $x$ we have an isomorphism 
\[ 
L_{\StableModuli_{0,0}(X, d^{-1} \beta) / \StableModuli_{0,0}(S, \beta), x}
\cong \mathcal{O}_{\StableModuli_{0,0}(X, d^{-1} \beta), x}
\]
 and we know that $H^{0}((R\pi_*Lf^*\MacBeth)^\vee)$ is a rank one torsion free sheaf with $\delta$ is surjective. Therefore we have the following situation: a local ring $R$, a rank one torsion free $R$-module $M$ and a surjection $\delta:M \rightarrow R$. We claim that this implies that $M$ is isomorphic to $R$. Since $\gamma$ is surjective we can choose $m \in M$ mapping to $1 \in R$. Suppose that there were $n \in M$ not in the submodule of $M$ generated by $m$. The element $n$ maps to $f \in R$ and so at every generic point of $R$ one has $n = fm$. But $M$ is torsion free and so $n = fm$. Therefore $m$ generates $M$ freely, i.e., $R \cong M$ as $R$-modules. Therefore $H^{0}((R\pi_*Lf^*\MacBeth)^\vee)$ is isomorphic to $ram^* L_{\PP^1/\PP^2}$.

\Lawrence{The Gysin map $\iota^!$ produces a class on $[\mathbb{P}^1/C_2] \times_{\MTwoOne} \StableModuli_{0,0} (S, \beta)$. However the moduli space $\StableModuli_{0,0} (X, d^{-1} \beta)$ is a two to one \'etale cover of $[\mathbb{P}^1/C_2] \times_{\MTwoOne} \StableModuli_{0,0} (S, \beta)$. Therefore by ~\cite{VirtualPushforward} we obtain the formula
\[ 2 \iota^! [\StableModuli_{0,0} (S, \beta)]^{virt} = \iota_* [\StableModuli_{0,0} (X, d^{-1} \beta)]^{virt}.\]
We call this \emph{the comparison formula for bisections on $S$ and sections on $X$}.}

\end{proof}

This leaves the question of relating the relative invariants on $S$ to the classical invariants. Bryan and Leung studied this problem for the tangent order one case. Their motivation was to apply point conditions to stable maps directly by restricting to those stable maps whose marked points map to $P$. The deformation theory then needs to be modified, restricting to those deformations which vanish at the marked point. In their Appendix A they prove that this produces the correct invariants and extend this to the case of certain divisors. We apply their ideas here, starting with the easiest case. Recall that we write $\LogModuli (S^\dagger, \beta)$ for the moduli of log stable maps to $S^\dagger$ in a class $\beta$ tangent to the boundary at a single point with maximal tangency. 

\begin{theorem}[Log invariants of sections are restrictions of Gromov-Witten invariants]

Let \Lawrence{$E$} be a section of a rational elliptic surface $S$, and let $\beta = E + nF$. Then there is an inclusion of components $\iota:  \LogModuli (S^\dagger, \beta) \rightarrow \StableModuli_{0,0} (S, \beta) $ and this induces an equality of virtual classes $\iota^* [\StableModuli_{0,0} (S, \beta)] =  [ \LogModuli (S^\dagger, \beta) ] $.
\label{thm:LogEqualityForSections}

\begin{proof}

In~\cite{TheEnumerativeGeometryOfK3SurfacesAndModularForms} the authors constructed the moduli space $\StableModuli_{0,0} (S, \beta)$: any curve is the union of the section $E$ together with different covers of the rational fibres. Since any curve in $\LogModuli(S, \beta)$ has a unique choice of marked point $\sigma$ the forgetful map is injective and since different components correspond only to different covers of the rational fibres this map is an inclusion of components.
Indeed by Proposition 4.3 of~\cite{TheTropicalVertex} the components of
$\LogModuli(S,\beta)$ are precisely those components of $\StableModuli_{0,0}
(S,\beta)$ where no component of the universal curve maps into the boundary.

Let us denote by $\mathfrak{M}_{0,0}$ the moduli Artin stack of prestable
genus zero curves and by $\mathfrak{M}^{\dagger}_{0,1}$ the moduli Artin
stack of all pre-stable genus one-pointed log curves. Recall that
the virtual fundamental class of $\StableModuli_{0,0}(S,\beta)$ is
defined via an obstruction theory relative to $\mathfrak{M}_{0,0}$
and recall
from \cite{LogarithmicGromovWittenInvariants} that the virtual
fundamental class of the moduli space $\LogModuli(S,\beta)$ is defined
by an obstruction theory relative to $\mathfrak{M}^{\dagger}_{0,1}$.

We consider the product $\StableModuli (S, \beta)_{0,1} \times_S F_0$ and restrict to those components of
$\StableModuli (S, \beta)_{0,1}$ where no component of the domain curve maps into $F_0$. We have inclusions $\LogModuli(S, \beta) \subset \StableModuli (S, \beta)_{0,1} \times_S F_0 \cong \StableModuli (S, \beta)_{0,0}$. Furthermore there is a virtual class on $\StableModuli (S, \beta)_{0,1} \times_S F_0$ relative to $\mathfrak{M}_{0,1}$ defined by the obstruction bundle induced by the pull-back
of the kernel of the natural morphism $\Theta_{S/k} \rightarrow{\mathcal N}_{F_0/S}$. Note that this kernel is rank two bundle away from the singular
points of $F_0$, and no curve in the moduli space passes through these
singular points.
Sections of the pull-back of this bundle then correspond to infinitesimal
deformations that deform the marked point in a direction along $F_0$.
It is a folklore result exposited in remark A.5 of the appendix of~\cite{TheEnumerativeGeometryOfK3SurfacesAndModularForms} that the virtual class this defines is equal to the pullback of the virtual class on $\StableModuli (S, \beta)_{0,0}$. \Lawrence{Their result holds so long as $F_0$ and $S$ are both smooth. This is not the case and we instead consider $\Spec k \in \PP^1$ which indeed does have smooth source and target. Then the Gysin map of $F_0$ in $S$ is compatible with the Gysin map of a point in $\PP^1$ for which we do have the desired equality.}

It remains to compare the virtual classes on $\StableModuli (S, \beta)_{0,1} \times_S F_0$ and on $\LogModuli(S, \beta)$. Write $D_i$ for the components of $F_0$. There is a standard exact sequence of sheaves on $S$ given by
\[0 \rightarrow \Theta_{S^\dagger/k}^\dagger \rightarrow \Theta_{S/k} \rightarrow \bigoplus \Lawrence{i_*} \mathcal{N}_{D_i/S} \rightarrow 0.\]
By definition the curves we wish to count do not map into the intersections $D_i \cap D_j$ and so the pull-backs of $\Theta^{\dagger}_{S^{\dagger}/k}$ or
$\ker(\Theta_{S/k}\rightarrow {\mathcal N}_{F_0/S})$ to any stable curve in $\LogModuli(S, \beta) \cong \StableModuli (S, \beta)_{0,1}$ are isomorphic. We consider the two cotangent complexes $L_{\LogModuli (S^{\dagger}, \beta) / \mathfrak{M}^{\dagger}_{0,1}}$ and $L_{\StableModuli_{0,1} (S, \beta) \times_S F_0 / \mathfrak{M}_{0,1}}$. Note the log structure on $\LogModuli (S^{\dagger}, \beta)$
is the pull-back of the basic log structure on $\mathfrak{M}_{0,1}$ under
the forgetful map because there are no components or nodes of the domain
curve mapping into $F_0$. Thus we have a strict factorisation 
$\LogModuli(S^{\dagger},\beta)\rightarrow \mathfrak{M}_{0,1}
\rightarrow \mathfrak{M}_{0,1}^{\dagger}$, with the second morphism
\'etale. Thus the two cotangent complexes are isomorphic.

Since we have an isomorphism of obstruction theories and an isomorphism of underlying schemes the corresponding virtual classes are equal.

\end{proof}

\end{theorem}

The techniques of Bryan and Leung do not apply directly to the tangency order two case. This means we have to work harder.

\begin{theorem}[Log invariants of bisections are restrictions of Gromov-Witten invariants]

Let $ i : \mathbb{P}^1 \rightarrow \Hilb \cong \mathbb{P}^2/C_2$ be the hyperplane of double covers ramifying over $0 \in \mathbb{P}^1$ and $\beta$ be a curve class with $\beta . K_S = -2$. We denote the product $\StableModuli_{0,0} (S,\beta) \times_{\Lawrence{[\mathbb{P}^2/C_2]}} \mathbb{P}^1$ by $\StableModuli^{ram}(S, \beta)$ with projection maps $pr$ to $\StableModuli_{0,0} (S, \beta)$ and $ram$ to $\mathbb{P}^1$. Then the moduli space of log stable maps to $S^\dagger$, $\LogModuli (S^\dagger,\beta)$ admits a morphism $j$ to $\StableModuli^{ram} (S,\beta)$ which is an inclusion of components and there is an equality of virtual classes

\begin{equation}
j^* i^! [\StableModuli_{0,0}(S,\beta)]^{virt} = [\LogModuli(S^\dagger,\beta)]^{virt}
\end{equation}

\label{thm:logtocclassicalcomparison}

\begin{proof}

The map $j$ is the forgetful map, forgetting the choice of marking. A log curve in this space only meets the boundary in a single point, so maps to the desired space. By construction of the space $\overline{\mathcal M}^{ram} (S, \beta)$ there is a canonical choice of marked point given by intersecting with the boundary $F_0$ so long as no component of the curve maps into $F_0$. Viewing this
point as a logarithmic marked point with contact order $2$ with one of the
boundary divisors $D_i$ or contact orders $1$ with adjacent divisors
$D_i$, $D_j$, 
we obtain a family of stable log maps. Thus $\LogModuli (S^\dagger, \beta)$ is open inside some components of $\StableModuli^{ram}$ but is also proper and hence closed. Thus $j$ is an inclusion of components. To be precise it is those components which do not cover the marked rational fibre and have precisely one point mapping to the boundary, which excludes any point components. 

As in the single section case we relate potentially distinct obstruction theories on $\LogModuli (S^\dagger, \beta)$ and $\StableModuli_{0,0}(S,\beta)$. Recall that the obstruction theories on $\LogModuli (S^\dagger, \beta)$ and $\StableModuli_{0,0}(S,\beta)$ are given by $R \pi_{S, *} f_S^* \Theta_{S^\dagger/k} ^{\dagger \: \vee}$ and $R \pi_{S, *} f_S^* \Theta_{S/k} ^{\vee}$ respectively. There is an exact sequence
\[ 0 \rightarrow  \Theta_{S^\dagger/k}^\dagger \rightarrow \Theta_{S/k} \rightarrow \bigoplus \mathcal{O}_{D_i}(D_i) \rightarrow 0 \]
where $F_0 = \bigcup D_i$. Pulling back to the universal curve over $\LogModuli (S^\dagger, \beta)$, pushing forward along $\pi$ and dualising we therefore obtain a distinguished triangle
\[ (R \pi_* f^* \bigoplus \mathcal{O}_{D_i}(D_i))^\vee \rightarrow (R \pi_*f^*\Theta_{S/k})^\vee \rightarrow (R \pi_*f^*\Theta_{S^\dagger/k}^\dagger)^\vee \rightarrow\]
Now $f^* \bigoplus \mathcal{O}_{D_i}(D_i)$ is supported only at a single point of each fibre of $\pi$, hence it pushes forwards to a line bundle on $\LogModuli (S^\dagger, \beta)$. Consider the cotangent triangle associated to the maps $\LogModuli (S^\dagger, \beta) \hookrightarrow \StableModuli (S, \beta) \rightarrow \mathfrak{M}_{0,0}$, noting that as before the map $\LogModuli (S^\dagger, \beta) \rightarrow \mathfrak{M}_{0,0}^\dagger$ factors through the \'etale map $\mathfrak{M}_{0,0} \rightarrow \mathfrak{M}_{0,0}^\dagger$
\[ i^*L_{\StableModuli (S, \beta)/\mathfrak{M}_{0,0}} \rightarrow L_{\LogModuli (S^\dagger, \beta)/\mathfrak{M}_{0,0}^\dagger} \rightarrow L_{\LogModuli (S^\dagger, \beta)/ \StableModuli (S, \beta)} \rightarrow\]
Since $(R \pi_*f^*\Theta_{S/k})^\vee$ and $(R \pi_*f^*\Theta_{S^\dagger/k}^\dagger)^\vee$ are obstruction theories we obtain by~\cite{VirtualPullbacks} that $(R \pi_* f^* \bigoplus \mathcal{O}_{D_i}(D_i))^\vee[1]$ is an obstruction theory for $L_{\LogModuli (S^\dagger, \beta)/ \StableModuli (S, \beta)}$. But $L_{\LogModuli (S^\dagger, \beta)/ \StableModuli (S, \beta)}$ is a regular immersion defined by one equation, namely that defining $\PP^1 \subset \PP^2$. We therefore deduce that $L_{\LogModuli (S^\dagger, \beta)/ \StableModuli (S, \beta)}$ is a line bundle supported in degree minus one. But $(R \pi_* f^* \bigoplus \mathcal{O}_{D_i}(D_i))^\vee[1]$ is also a line bundle supported in degree minus \Lawrence{one} and maps surjectively and hence isomorphically to $L_{\LogModuli (S^\dagger, \beta)/ \StableModuli (S, \beta)}$. This produces an equality of virtual classes on $\LogModuli(S^\dagger, \beta)$
\[\iota^! ([\StableModuli_{0,0}(S,\beta)]^{vir}) = [\LogModuli(S,\beta)]^{vir}\]
as desired.
\end{proof}

\end{theorem}

We can combine this with our relation between the invariants of $X$ and $S$, producing the following theorem.

\begin{corollary}[A formula for bisections on $S$]

There are equalities
\[  deg \: (ram^* \Lawrence{(2H)} \cdot[\StableModuli_{0,0} (S, \beta)]^{virt}) = deg \:([\StableModuli_{0,0} (X, d^{-1}\beta)]^{virt}) \]
and
\[deg \: (j^* (ram^* \Lawrence{(H)} \cdot [\StableModuli_{0,0} (S, \beta)]^{virt})) = deg \: ([\LogModuli (S^\dagger, \beta)]^{virt}) \]
\end{corollary}

\section{Bounding Gromov-Witten invariants}

As we said at the start of the previous section there is no reason why when we pass to the log moduli space we don't remove components with high degree, potentially ruining our bound of the $J$-function. This section is dedicated to proving that restriction of the virtual class to curve components is positive, whilst the restriction to bubble components is not too negative. We begin by proving that the point components make no contribution to either of our invariants.

\begin{lemma}[Point components are trivial]

The restriction of $ram^* H \cdot [\StableModuli_{0,0} (S, \beta)]^{virt}$ to point components is zero.

\end{lemma}

\begin{proof}

This follows from conservation of number and choosing a different hyperplane to intersect with.

\end{proof}

By relating the invariants on $S$ to those on a $K3$ fibre of $X$ we can apply techniques from~\cite{TheEnumerativeGeometryOfK3SurfacesAndModularForms}. The authors there prove positivity of the virtual class restricted to various components of the moduli space.

\begin{theorem}[Curve components are positive]

The restriction of $ram^* H \cdot [\StableModuli_{0,0} (S, \beta)]^{virt}$ to the preimage of a curve component $\mathcal{M}$ is positive.

\label{thm:CurvesArePositive}

\begin{proof}
  \Lawrence{By our compatibility result relating Gromov-Witten invariants on $S$ and $X$ it is enough to prove this positivity on the threefold $X$}. This problem was studied by Maulik and Pandharipande in~\cite{GromovWittenTheoryAndNoetherLefschetzTheory} but we can prove the equality here via a direct calculation. Recall that the reduced invariants on a $K3$ surface $K$ are given by an obstruction theory $T^\circ$ fitting into an exact triangle
  \[  \tau_{\geq -1} R \pi _* \omega_{\pi} \otimes H^0 (K, \omega_K) \rightarrow R \pi_* f^* \Theta_{K/k} \rightarrow T^\circ \rightarrow\]
    where $\tau_{\geq -1} R \pi _* \omega_{\pi} \otimes H^0 (K, \omega_K)$ is a line bundle supported in degree minus one. Suppose now that $i:K \hookrightarrow X$ is a fibre of the $K3$ fibration of $X$ into which all the stable curves in $\mathcal{M}$ map. There is an exact sequence on $K$
    \[ 0 \rightarrow \Theta_{K/k} \rightarrow \Lawrence{i^*} \Theta_{X/k} \rightarrow \mathcal{N}_{K/X} \rightarrow 0 \]
    where the normal bundle $\mathcal{N}_{K/X}$ is trivial since $K$ is a fibre. Pulling this back to the universal curve, pushing forwards to the moduli space and dualising we obtain a distinguished triangle on $\StableModuli_{0,0} (X, \beta)$
    \[ \mathcal{L} \rightarrow (R \pi_* f^* i^* \Theta_{X/k})^\vee \rightarrow (R \pi_* f^* \Theta_{K/k})^\vee  \rightarrow \]
    where $\mathcal{L}$ is a line bundle supported in degree zero. By composition we find a morphism $(R \pi_* f^* i^* \Theta_{X/k})^\vee \rightarrow T^\circ$ and let $\Cone$ be the cone over this. Applying the octahedral axiom we obtain a diagram
 \[ \begin{tikzpicture}
      \node (T) at (0,0) {$T^\bullet$};
      \node (L) at (0,2) {$\mathcal{L}[1]$};
      \node (C) at (3,0.5) {$\Cone[1]$};
      \node (K) at (-3,0.5) {$R \pi_* f^* \Theta_{K/k}^\vee$};
      \node (X) at (-6,-1) {$R \pi_* f^* \Theta_{X/k}^\vee$};
      \node (R) at (6,-1) {$\tau _{\geq -1} R \pi _* \omega_{\pi} \otimes H^0 (K, \omega_K)[1]$};
      \node (E0) at (1.5,2.75) {};
      \node (E1) at (4.8,0.8) {};
      \node (E2) at (9,-1.5) {};
      \node (E3) at (9,-2.5) {};
      
      \draw [->] (L) -- (E0);
      \draw [->] (C) -- (E1);
      \draw [->] (R) -- (E2);
      \draw [->] (R) -- (E3);      
      \draw [->] (L) -- (C);
      \draw [->] (K) -- (L);
      \draw [->] (K) -- (T);
      \draw [->] (T) -- (C);
      \draw [->] (X) -- (K);
      \draw [->] (X) -- (T);
      \draw [->] (T) -- (R);
      \draw [->] (C) -- (R);
      \end{tikzpicture}\]
    showing that the following triangle is distinguished
    \[  \mathcal{L} \rightarrow \Cone \rightarrow  \tau _{\geq -1}R \pi _* \omega_{\pi} \otimes H^0 (K, \omega_K) \rightarrow \]
    Now the special property of the preimage of a curve components is that each component is the moduli space of stable maps to $K$. In particular the degree zero cohomology of $(R \pi_* f^* i^* \Theta_{X/k})^\vee$ and $T^\circ$ are isomorphic since both are isomorphic to the tangent space of $\mathcal{M}$. This is enough to show that in fact $\Cone$ vanishes since $\Cone$ has a two term resolution by line bundles, \Lawrence{the differential is a surjection of line bundles since the degree zero cohomology vanishes and hence this surjection is an isomorphism.}

    Now $\mathcal{M}$ is a component of the moduli space of stable maps both to $K$ and $X$ and the two induced obstruction theories are isomorphic. The result is known from~\cite{TheEnumerativeGeometryOfK3SurfacesAndModularForms} that each component of the moduli space contributes either one or zero to the total count.

\end{proof}

\end{theorem}

Now on the bubble component we prove that the virtual class cannot become too negative.

\begin{theorem}[Bubble components are bounded]
  
Let $X$ be a generic unravelling of $S$, $\StableModuli^{bub} \subset \StableModuli_{0,0} (X, d^{-1}\beta)$ the preimage of a bubble component, a family over $\iota : \mathbb{P}^1 \rightarrow \Hilb \cong \mathbb{P}^2/C_2$. Choosing a point $P: Spec \: k/C_2 \rightarrow \mathbb{P}^2/C_2$ in the image of $\iota$ defines a $K3$ surface $cov:K \rightarrow S$. This induces a map $P: \StableModuli_{0,0} (K, cov ^* \beta) \rightarrow \StableModuli^{bub}$ and generically fibres $\StableModuli^{bub}$ by moduli spaces of stable maps to $K3$ surfaces, $Kl: \StableModuli^{bub} \rightarrow \mathbb{P}^1$. We claim that there is a constant $q \in \mathbb{Q}$ with $P_* [\StableModuli_{0,0} (K, cov ^* \beta)] = q [_{0,0} (X, d^{-1} \beta)]$ and that $q$ is dependent only on the choice of $X$.\label{thm:BubblesArentTooNegative}
\begin{proof}

  This question is related to the comparison theory of Maulik and Pandharipande in~\cite{GromovWittenTheoryAndNoetherLefschetzTheory}. In~\cite{GromovWittenTheoryAndNoetherLefschetzTheory}, the authors consider a family of nonsingular $K3$ over a non-singular curve $C$, $p: X \rightarrow C$. The Gromov-Witten invariants of $X$ are related to those of a fibre $K$ by multiplication by the degree of the Hodge bundle of $X \rightarrow C$. The details of this calculation are contained in~\cite{HodgeIntegralsAndGromovWittenTheory}.

  In our situation however there are some singular fibres of the fibration $p: X \rightarrow \PP^1$. Twenty four of these fibres are ordinary double points, occurring where the double cover $\PP^1 \rightarrow \PP^1$ ramifies over the image of a rational fibre. One can take a resolution of this family as described in Example 5.1 of~\cite{GromovWittenTheoryAndNoetherLefschetzTheory} to obtain a new family of non-singular $K3$ surfaces with the same Gromov-Witten invariants. In their paper Maulik and Pandharipande move to an analytic space to construct the small resolution. \Lawrence{This is a strictly analytic construction and to apply techniques of log geometry we must remain in the algebraic world. We appeal to Theorem 7.3 of~\cite{AlgebraizationOfFormalModuliII} which says that any Moishezon manifold is the analytification of an algebraic space. This resolution is indeed a Moishezon manifold since it is a small resolution of another Moishezon manifold. Therefore we consider a small resolution $\pi: \tilde{X} \rightarrow X$ with $\tilde{X}$ a smooth algebraic space. The moduli stack of stable maps to a Deligne Mumford stack was constructed in~\cite{GromovWittenTheoryOfDeligneMumfordStacks}. In that paper the authors also proved that the usual technology for constructing virtual classes can be applied to such objects and so the techniques of~\cite{HodgeIntegralsAndGromovWittenTheory} continue to apply.} However there are also two singular fibres where the $K3$ surface degenerates to the union of two rational elliptic surfaces meeting along a smooth elliptic fibre. We augment this singular family to a smooth log family. Let $K_1$ and $K_2$ denote these two singular fibres and we write $X^\dagger$ for $X$ with the divisorial log structure induced by the $K_i$. Since we have not imposed any tangency conditions \Lawrence{and a generic curve is transverse} the classical and log Gromov-Witten invariants coincide. Now by construction the map $p: X^\dagger \rightarrow \PP^1$ is log smooth and we can define a relative theory by taking the cone over the natural morphism
    \[\tau_{\geq-1} R \pi_* (\omega_\pi^\dagger) ^\vee \otimes H^0 (X, \omega_{X/\PP^1}) \rightarrow R \pi_* f^* (\Theta_{X/\PP^1})\]
    This gives a relative obstruction theory which may play the role of the relative theory defined in section 2.2 of~\cite{GromovWittenTheoryAndNoetherLefschetzTheory}. In particular their equation
    \[ [\StableModuli _{0} (\pi, \epsilon)]^{vir} = c_1 (K^*) \cap [\StableModuli_{0}(\pi, \epsilon)]^{red}\]
from page 22 continues to hold, but where $K$ is the log canonical line bundle with fibre $H^0 (X_{\xi}, K^\dagger_{X_\xi})$ over a point $\xi$. This proves the desired result since we may take the Chern class to be supported away from the singular fibres.

\end{proof}

\end{theorem}

\section{Assembling the family}

We now have a description of the product formula on the mirror, of the terms appearing in the scattering diagram and of the piecewise linear function $\phi$. Therefore we can assemble all of these together to prove convergence of the mirror. This pits the growth of the function $\phi$ against the growth of the number of curves, so we begin with bounding the growth of $\phi$.

\begin{lemma}[Bounding the function $\phi$]
  We choose the lift of $\phi$ which is zero on the cone $C$ of Figure~\ref{fig:DualIntersectionRES}. Over a one-cell $v_i$ recall that the piecewise linear function $\phi$ changes by $D_i \otimes n_i$, where $n_i$ is defined in Definition
\ref{def:sheaf-of-monoids}. Then $\phi$ is given by the following formulae on the following cones. On $\langle (4n, 1), (4n+1, 1)\rangle$ it is given by
\[nx F - (n(2n-1)D_1 + 2n^2D_2 + n(2n+1)D_3 + n(2n+2)D_4)y\]
whilst on $\langle (4n+1, 1), (4n+2, 1)\rangle$ it is given by
\[nx F + x D_1 - ((n+1)(2n+1)D_1 + 2n^2D_2 + n(2n+1)D_3 + n(2n+2)D_4)y\]
and on $\langle (4n+2, 1), (4n+3, 1)\rangle$ it is given by
\[nx F + x(D_1 + D_2) - ((n+1)(2n+1)D_1 + 2(n+1)^2D_2 + n(2n+1)D_3 + n(2n+2)D_4)y\]
and finally on $\langle (4n+3, 1), (4n+4, 1)\rangle$ it is given by
\[nx F + x(D_1 + D_2 + D_3) - ((n+1)(2n+1)D_1 + 2(n+1)^2D_2 + (n+1)(2n+3)D_3 + n(2n+2)D_4)y\]

We now find rough bounds on $\phi$. Suppose that $m > n > 0$ are integers and we wish to compute $\phi ((m,1)) - \phi ((n,1))$. Since $\phi$ is strictly convex the difference between these is bounded by $\phi((m,1)) - \phi((m-1,1))$. From the formulae above we see that $\phi ((m,1)) - \phi ((m-1,1)) \geq (2\left \lfloor{m/4}\right \rfloor -1) D_{m-1}$ where the index of $D_{m-1}$ is taken mod $4$. The inequality $\phi ((m,1)) - \phi ((m-1,1)) \geq (1- 2\left \lfloor{m/4}\right \rfloor) D_{m+1}$ similarly holds for $m < n < 0$. Now suppose that $n,k > 0$ and we consider $\phi ((m,1)) - \phi ((n,1)) - \phi ((k,1))$ with $m \ge n+k$. By convexity this value is smallest if $n = m-1$ and $k=1$. In that case we can use the above bound to see that this is bounded also by $(2\left \lfloor{m/4}\right \rfloor -1)D_{m-1}$.

\end{lemma}

Recall the structure of the equations from Table~\ref{eqntable}. We will work through each term appearing and prove that they all converge in a neighbourhood of the origin.

\begin{proposition}

The coefficients $f_{(k,2)}, g_{(k,2)}$ and $r_{(k,2)}^i$ of these equations converge in a neighbourhood of the large complex structure limit point.

\begin{proof}

  By symmetry it is enough to prove convergence only of the first equation and of one of the products $\vartheta_{D_i}^2$.

  The coefficient $f_{2,2}$ is the sum over pairs which do not bend anywhere and end close to $(2,2)$. The pairs of pants are pairs of broken lines from $(4n,1)$ and $(-4n+2,1)$. Near $(2,2)$ these lines carry the monomials $z^{\phi (4n,1)}$ and $z^{\phi (-4n+2,1)}$. Therefore we want to study the convergence of
  \[\sum |z^{\phi (4n,1) + \phi (-4n+2,1)}| < \sum |z^{n^2F}|\]
  which converges since it decays at least exponentially. Exactly the same argument applies to show the convergence of $f_{(4,2)}$, $r_i^i$ and $r_{i+2}^i$. In all these cases none of the broken lines can bend, there is a $\ZZ$ indexed family of pairs of pants and the monomials are controlled entirely by the function $\phi$. For the product $\vartheta_{D_i}^2$ there is precisely one pair of broken lines which lie in only one maximal cell of $B$. This explains the leading coefficient of $1$ appearing in the formulae.

\end{proof}

\end{proposition}
We now study the contribution to the product of points with $E(P) = 1$, i.e.,
the theta functions $\vartheta_{D_i}$, restricting our analysis by symmetry to $\vartheta_{D_1}\vartheta_{D_3}$ and $\vartheta_{D_4}^2$. Let us begin with $\vartheta_{D_1}^2$ where we describe all pairs of broken lines which contribute to $\vartheta_{D_4}$. All such pairs contain one bend which for the purpose of proving convergence we can assume occurs on the first broken line. We may also assume that the first broken line is from $(4m + 4n, 1)$ with $m, n$ positive and scatters off of $(4n, 1)$ with a monomial $z^{E+kF}$ for $E$ a section meeting $D_1$ and $k$ non-negative. As a consequence the second line is from $(-4m, 1)$. This pair of lines contributes $I_{0,0,E+kF}^{\dagger}z^{E + \phi ((4m+4n,1)) - \phi ((4n,1)) + \phi ((-4m,1)) + kF}$ to the coefficient. This sum is taken over all $m,n > 0$, $k \geq 0$ and sections $S$ meeting $D_1$ in one point. 

\begin{lemma}[Single bends converge]

The following sum converges
\[\sum_{m,n,k, S} I_{0,0,S+K}^{\dagger}z^{S + \phi ((4m + 4n, 1))  + \phi ((-4n,1)) - \phi ((4m,1)) + kF}\]
\begin{proof}
We rewrite this sum as
\begin{equation}
\left( \sum _{S,k} I_{0,0,S+kF}^{\dagger} z^{S + kF}  \right) \left( \sum_{m,n}  z^{\phi ((4m + 4n, 1)) + \phi ((-4n,1)) - \phi ((4m,1))}\right)
\end{equation}

We know that the term $I_{0,0,S+kF}^{\dagger}$ has modulus at most $2^k$, therefore for $| z^F | < 1/2$ and $| z^S | < 1$ and after applying Lemma~\ref{thm:BoundedlyManySections} one has convergence of the left hand sum.

To prove convergence of the right hand sum we expand out the sum using the explicit formula for $\phi$ found above and bound it
\[\sum_{m,n} | z^{\phi ((4m+4n,1)) - \phi ((4m,1)) + \phi ((-4n,1))} | \le \sum_{m,n} | z^{\phi ((-4n,1)) + ((2m+2n-1)D_{3})} |\]
And this converges since $\phi ((-4n,1))$ is positive and grows quadratically in $n$.

\end{proof}

\end{lemma}

Therefore we have shown that all the functions appearing in the defining equations of the mirror family are in fact holomorphic in a neighbourhood of the origin.

\section{Recognising the family}

The above convergence result does not tell us very much about the family that one obtains. Fortunately the geometry of the singular locus allows us to say much more. Consider the equations~\ref{eqn:MirrorToRES1}-\ref{eqn:MirrorToRES3}, only looking at the terms of order two these simplify to:

\begin{align}
  \vartheta_{D_1} \vartheta_{D_3} & = f_{(2,2)} \vartheta_{2D_2} + f_{(6,2)} \vartheta_{2D_4}    \label{eqn:ReducedMirrorToRES1} \\
  \vartheta_{D_2} \vartheta_{D_4} & = f_{(0,2)} \vartheta_{2D_1} + f_{(4,2)} \vartheta_{2D_3}    \label{eqn:ReducedMirrorToRES2}\\
  \vartheta_{2D_i} & = \frac{\vartheta_{D_i}^2 (1+r_{i+2,2}^{i+2})-r_{i+2,2}^i \vartheta_{D_{i+2}}^2} {(1+r_{i,2}^i)(1+r_{i+2,2}^{i+2}) - r_{i,2}^{i+2} r_{i+2,2}^{i}}\label{eqn:ReducedMirrorToRES3}
\end{align}
These equations correspond to taking the projective closure of the above family and looking at the fibre at infinity. Substituting in the final equation into the first two we obtain the following pair of quadrics:

\begin{align}
  \vartheta_{D_1} \vartheta_{D_3} & = \frac{f_{(2,2)}(1+r_{4,2}^4) - f_{(6,2)}r^4_{2,2}}{(1+r^2_{2,2})(1+r^4_{4,2})-r^4_{2,2}r^2_{4,2}} \vartheta_{D_2}^2 + \frac{f_{(6,2)}(1+r_{2,2}^2) - f_{(2,2)}r^2_{4,2}}{(1+r^2_{2,2})(1+r^4_{4,2})-r^4_{2,2}r^2_{4,2}} \vartheta_{D_4}^2    \label{eqn:Ugh1} \\
  \vartheta_{D_2} \vartheta_{D_4} & = \frac{f_{(0,2)}(1+r_{3,2}^3) - f_{(4,2)}r^3_{1,2}}{(1+r^1_{1,2})(1+r^3_{3,2})-r^3_{1,2}r^1_{3,2}} \vartheta_{D_1}^2 + \frac{f_{(4,2)}(1+r_{1,2}^1) - f_{(4,2)}r^1_{3,2}}{(1+r^1_{1,2})(1+r^3_{3,2})-r^3_{1,2}r^1_{3,2}} \vartheta_{D_3}^2 \label{eqn:Ugh2}
\end{align}

There are some relations between the functions involved in this description. To simplify this description let us restrict to the locus where the areas of each of the $D_i$ are all $v$. Let us show that this curve is smooth over this locus. We will also be interested in some questions of modularity of certain constants, and we begin with generalities of equations of the form:

\[ X_1  X_3 = tX_2 ^2 + t X_4 ^2, \quad X_2  X_4= tX_1 ^2 + t X_3 ^2\]

Putting these equations into Sage we can arrange them into Weierstrass form (using the WeierstrassForm function). This produces the equation

\[v^2 - u^3 - u(-t^8/3 - 7t^4/24 - 1/768) - (2t^{12}/27 - 11t^8/72 - 11t^4/1152 + 1/55296)\]
A generically smooth elliptic curve with $j$-invariant
\begin{align}
\frac{16777216t^{24} + 44040192t^{20} + 38731776t^{16} + 11583488t^{12} + 151296t^8 + 672t^4 + 1}{65536t^{20} - 16384t^{16} + 1536t^{12} - 64t^8 + t^4}
\label{jinvariant} \end{align}
Making the substitution $s = 4t^2$ we obtain the following expression.
\[16 \frac{(s^4 + 14 s^2 + 1)^3} {s^{2} (s -1)^4 (s + 1)^4}\]
We compare this to the formula for the $j$-invariant in terms of the Jacobi modulus $k$. By definition the $j$-invariant is given by
\[ 256 \frac{(k^4-k^2 + 1)^3}{k^4 (k^2-1)^2} \]
After a change of coordinates $k = \frac{u + 1/4}{\sqrt{u}}$ we find that this is equal to
\[256 \frac{(u^4 + 7u^2/8 + 1/256)^3}{(u^{8} + (-1/4)u^{6} + 3/128u^4 + (-1/1024)u^2 + 1/65536)u^2}\]
and so lining this up with equation~\ref{jinvariant} we find that $u = s/4$. Substituting backwards we find $k = \frac{t^2 + 1/4}{t}$.

Now let us evaluate the function $t$ for the locus where the $z^{[D_i]}$ are all equal and we express the functions in terms of the area of a general smooth fibre, say $e^{i\pi \rho}$.

\[f_{(k,2)} (v) = \sum_{n} v^{(4n+1)^2} = \Theta_2 (0,\rho)\]
\[1 + r_{i,2}^i (v) = \sum_{n \: even} v^{4n^2} = (\Theta_3 (0,\rho) + \Theta_4(0,\rho))/2\]
and
\[r_{i+2,2}^i (v) = \sum_{n \: odd} v^{4n^2} = (\Theta_3 (0,\rho) - \Theta_4(0,\rho))/2\]
where $\Theta_i (z,\rho)$ are the Jacobi theta functions. Substituting these into equations~\ref{eqn:Ugh1} and~\ref{eqn:Ugh2} we find that $t$ is equal to the ratio $\frac{\Theta_2(0,\rho)}{2 \Theta_3 (0,\rho)}$. In particular this is not constant in $\rho$ and so for generic choices is smooth. Furthermore the Jacobi modulus $k$ of this curve is equal to $\frac{\Theta_2 (0,\rho) ^2 + \Theta_3 (0,\rho)^2}{2 \Theta_2 (0,\rho) \Theta_3 (0,\rho)}$. But by definition we have that $k$ is also $\Theta_2 (0,\tau) / \Theta_3 (0,\tau)$ where $\tau$ is the fundamental period of the elliptic curve. Of course there should be a relation between these two:

\begin{align*}
  \frac{\Theta_2 (0,\rho) ^2 + \Theta_3 (0,\rho)^2}{2 \Theta_2 (0,\rho) \Theta_3 (0,\rho)} & = \frac{1}{2} \left( \frac{\Theta_3 (0,\rho)}{\Theta_2 (0,\rho)} + \frac{\Theta_2 (0,\rho)}{\Theta_3 (0,\rho)}  \right) \\
  & = \frac{\Theta _3 (\rho /2 )^2} {\Theta_2 (\rho /2)^2}
  \end{align*}

The modular group acts on the Jacobi theta functions, with $\tau \mapsto \tau + 1$ swapping the pairs $\Theta_3$ and $\Theta_4$ and $\Theta_1$ and $\Theta_2$. The map $\tau \mapsto -1/\tau$ swaps $\Theta _2$ and $\Theta _4$ but preserves $\Theta_1$ and $\Theta_3$. Conjugating the first of these by the second we see that

\[ \frac{\Theta _3 (\rho /2 )^2} {\Theta_2 (\rho /2)^2} = \frac{\Theta _2 (\frac{\rho}{2-\rho})^2} {\Theta_3 (\frac{\rho}{2-\rho})^2}\]

But this should also be equal to $\frac{\Theta _2 (\tau)^2} {\Theta_3 (\tau)^2}$ and so we can deduce that $\tau$ and $\frac{\rho}{2-\rho}$ are equal up to conjugacy under $SL (2, \ZZ)$. But we have just seen that $\rho$ and $\frac{\rho}{2-\rho}$ are also conjugate. Therefore the two associated elliptic curves are isomorphic. The limit as the area of a curve approaches zero corresponds to the limit as the imaginary part of $\rho$ goes to positive infinity. In this limit the value of $\tau$ also approaches the cusp point, and so the corresponding elliptic curve degenerates to the Tate curve. Of course we have already seen this, in the limit that we turn off all corrections this elliptic curve degenerates to a cycle of four rational curves.

This construction connects very pleasantly to the story of mirror symmetry for elliptic curves as described by Dijkgraaf in~\cite{MirrorSymmetryAndEllipticCurves} or Polischuk and Zaslow in~\cite{CategoricalMirrorSymmetryTheEllipticCurve}. These describe the mirror to an elliptic curve with complexified K\"ahler class $\alpha$ as the elliptic curve with periods $\langle 1, \alpha \rangle$. We do not need to specify a complex structure on the first curve since the Fukaya category does not depend on this choice. The choice of volume $e^{i\pi \tau}$ corresponds to a choice of complexified K{\"a}hler class on the general fibre, and thus we see a philosophical explanation of the above equality.

If we restrict now to the locus where $z^{[C]}$ vanishes for $[C].[F] > 0$ the above surface degenerates further, only the above terms appear in the defining equation. Since now the equation is homogenous in affine space we recover the family as being a deformation of a cone over the mirror elliptic curve.

\section*{Future questions}

There are several questions that immediately spring from this project. Firstly what happens as one allows the areas of the different components of the fibre to vary. The elliptic curve at infinity deforms in a controlled manner. I do not believe that there is any reason to think that the periods should be as well behaved, but rather there will be some codimension one locus under which the previous description holds. It would be interesting to describe this locus. Secondly the reason we needed the $I_4$ fibre was simply to control the combinatorics and the relevant terms in the scattering diagram. Conjecturally one should be able to reconstruct the entire Gromov-Witten theory from our description. From this one should be able to describe the theory as one smooths the boundary to less degenerate configurations, whilst the generating functions remain holomorphic. If one could do this then the construction would converge relative to any boundary fibre. Thirdly, following the previous point, one could attempt to understand how to apply this construction to the central fibre of a type II degeneration, similarly to the work of Atsushi Kanazawa in~\cite{DoranHarderThompsonConjectureViaSYZMirrorSymmetryEllipticCurves} in the elliptic curve case.

\bibliographystyle{plain}
\bibliography{Draft}

\begin{thebibliography}{10}

\bibitem{GromovWittenTheoryOfDeligneMumfordStacks}
Dan Abramovich, Tom Graber, and Angelo Vistoli.
\newblock Gromov-{W}itten theory of {D}eligne-{M}umford stacks.
\newblock {\em Amer. J. Math.}, 130(5):1337--1398, 2008.

\bibitem{ComparisonTheoremsForGromovWittenInvariantsOfSmoothPairsAndOfDegenerations}
Dan Abramovich, Steffen Marcus, and Jonathan Wise.
\newblock Comparison theorems for {G}romov–{W}itten invariants of smooth
  pairs and of degenerations.
\newblock {\em Annales de l’institut Fourier}, 64(4):1611--1667, 2014.

\bibitem{AlgebraizationOfFormalModuliII}
M.~Artin.
\newblock Algebraization of formal moduli. {II}. {E}xistence of modifications.
\newblock {\em Ann. of Math. (2)}, 91:88--135, 1970.

\bibitem{AurouxKatzarkovOrlov}
Denis Auroux, Ludmil Katzarkov, and Dmitri Orlov.
\newblock {M}irror {S}ymmetry for del {P}ezzo surfaces: Vanishing cycles and
  coherent sheaves.
\newblock {\em Inventiones mathematicae}, 166(3):537--582, Dec 2006.

\bibitem{ExplicitEquationsForMirrorFamiliesToLogCalabiYauSurfaces}
L.~J. {Barrott}.
\newblock {Explicit equations for mirror families to log Calabi-Yau surfaces}.
\newblock {\em ArXiv e-prints}, October 2018.

\bibitem{GromovWittenInvariantsInAlgebraicGeometry}
K.~Behrend.
\newblock {G}romov-{W}itten invariants in algebraic geometry.
\newblock {\em Inventiones mathematicae}, 127(3):601--617, Sep 1997.

\bibitem{TheEnumerativeGeometryOfK3SurfacesAndModularForms}
Jim Bryan and Naichung~Conan Leung.
\newblock The enumerative geometry of {K3} surfaces and modular forms.
\newblock {\em Journal of the American Mathematical Society}, 13(2):371--410,
  2000.

\bibitem{MirrorSymmetryAndAlgebraicGeometry}
D.A. Cox and S.~Katz.
\newblock {\em Mirror Symmetry and Algebraic Geometry}.
\newblock Mathematical surveys and monographs. American Mathematical Society,
  1999.

\bibitem{EGA4}
Jean Dieudonn{\'e} and Alexander Grothendieck.
\newblock \'{E}l\'ements de g\'eom\'etrie alg\'ebrique.
\newblock {\em Inst. Hautes \'Etudes Sci. Publ. Math.}, 4, 8, 11, 17, 20, 24,
  28, 32, 1961--1967.

\bibitem{MirrorSymmetryAndEllipticCurves}
R.~Dijkgraaf.
\newblock Mirror {S}ymmetry and {E}lliptic {C}urves.
\newblock In {\em The {M}oduli {S}pace of {C}urves}, Progress in mathematics.
  Birkh{\"a}user, 1995.

\bibitem{HodgeIntegralsAndGromovWittenTheory}
C.~Faber and R.~Pandharipande.
\newblock {Hodge integrals and Gromov-Witten theory}.
\newblock {\em Inventiones mathematicae}, 139(1):173--199, 2000.

\bibitem{AMirrorTheoremforToricCompleteIntersection}
Alexander Givental.
\newblock {\em A Mirror Theorem for Toric Complete Intersections}, pages
  141--175.
\newblock Birkh{\"a}user Boston, Boston, MA, 1998.

\bibitem{GrossHackingSiebert}
M.~{Gross}, P.~{Hacking}, and B.~{Siebert}.
\newblock {Theta functions on varieties with effective anti-canonical class}.
\newblock {\em ArXiv e-prints}, January 2016.

\bibitem{MirrorSymmetryforLogCY1}
Mark Gross, Paul Hacking, and Sean Keel.
\newblock Mirror symmetry for log {C}alabi-{Y}au surfaces {I}.
\newblock {\em Publications math{\'e}matiques de l'IH{\'E}S}, 122(1):65--168,
  Nov 2015.

\bibitem{CanonicalBasesForClusterAlgebras}
Mark Gross, Paul Hacking, Sean Keel, and Maxim Kontsevich.
\newblock Canonical bases for cluster algebras.
\newblock {\em J. Amer. Math. Soc.}, 31(2):497--608, 2018.

\bibitem{TheTropicalVertex}
Mark Gross, Rahul Pandharipande, and Bernd Siebert.
\newblock The tropical vertex.
\newblock {\em Duke Math. J.}, 153(2):297--362, 06 2010.

\bibitem{LogarithmicGromovWittenInvariants}
Mark Gross and Bernd Siebert.
\newblock Logarithmic {G}romov-{W}itten invariants.
\newblock {\em J. Amer. Math. Soc.}, 26(2):451--510, 2013.

\bibitem{GrossSiebert}
Mark Gross and Bernd Siebert.
\newblock Theta functions and mirror symmetry.
\newblock In {\em Surveys in differential geometry 2016. {A}dvances in geometry
  and mathematical physics}, volume~21 of {\em Surv. Differ. Geom.}, pages
  95--138. Int. Press, Somerville, MA, 2016.

\bibitem{DoranHarderThompsonConjectureViaSYZMirrorSymmetryEllipticCurves}
Atsushi Kanazawa.
\newblock {D}oran-{H}arder-{T}hompson conjecture via {SYZ} mirror symmetry:
  Elliptic curves.
\newblock In {\em Symmetry, Integrability and Geometry: Methods and
  Applications}, volume~13. 12 2016.

\bibitem{DegenerationsOfK3SurfacesAndEnriquesSurfaces}
Vik.~S. Kulikov.
\newblock Degenerations of {$K3$} surfaces and {E}nriques surfaces.
\newblock {\em Izv. Akad. Nauk SSSR Ser. Mat.}, 41(5):1008--1042, 1199, 1977.

\bibitem{VirtualPullbacks}
Cristina Manolache.
\newblock Virtual pull-backs.
\newblock {\em J. Algebraic Geom.}, 21(2):201--245, 2012.

\bibitem{VirtualPushforward}
Cristina Manolache.
\newblock Virtual push-forwards.
\newblock {\em Geom. Topol.}, 16(4):2003--2036, 2012.

\bibitem{GromovWittenTheoryAndNoetherLefschetzTheory}
Davesh Maulik and Rahul Pandharipande.
\newblock Gromov-{W}itten theory and {N}oether-{L}efschetz theory.
\newblock In {\em A celebration of algebraic geometry}, volume~18 of {\em Clay
  Math. Proc.}, pages 469--507. Amer. Math. Soc., Providence, RI, 2013.

\bibitem{GeometryofEllipticSurfaces}
Rick Miranda.
\newblock {\em The Basic Theory of Elliptic Surfaces}.
\newblock Dipartimento di Matematica dell' Universita di Pisa, 1989.

\bibitem{CategoricalMirrorSymmetryTheEllipticCurve}
Alexander Polishchuk and Eric Zaslow.
\newblock Categorical mirror symmetry: The elliptic curve.
\newblock {\em Advances in Theoretical and Mathematical Physics},
  2(2):443--470, 1 1998.

\bibitem{StacksProject}
The {Stacks Project Authors}.
\newblock \textit{Stacks Project}.
\newblock \url{http://stacks.math.columbia.edu}, 2018.

\end{thebibliography}

\quad \newline
\noindent
\texttt{National Center for Theoretical Sciences\\ No. 1 Sec. 4
Roosevelt Rd., National Taiwan University\\ Taipei, 106, Taiwan}

\end{document}